\documentclass{article}

\rmfamily{\fontsize{12pt}{\baselineskip}\selectfont}
\usepackage{ amssymb }
\usepackage{color}
\usepackage{mathrsfs}
\usepackage{amssymb,latexsym}
  \usepackage{dsfont}
\usepackage{amsthm}
\usepackage{amsmath}
\usepackage{graphicx}
\usepackage{underscore}
\usepackage{subfig}
\usepackage {hyperref}
\usepackage{tikz}
\usepackage{pgfplots}

\theoremstyle{plain}
\newtheorem{theorem}{Theorem}[section]
\newtheorem{corollary}{Corollary}[section]

\newtheorem{claimd}{Claim}[section]
\newtheorem{prop}{Proposition}[section]

\theoremstyle{definition}
\newtheorem{assumption}{Assumption}

\allowdisplaybreaks
\def\eps{{\varepsilon}}

\def\R{\mathbb R}

\def\N{\mathbb N}

\def\ds{\displaystyle}

\author{J\'er\^ome Coville\thanks{
UR 546 Biostatistique et Processus  Spatiaux, INRA, Domaine St Paul Site Agroparc, F-84000 Avignon, France, {\itshape email:
}\ttfamily jerome.coville@inra.fr} , Changfeng Gui \thanks{One USA Circle,  Department of Mathematics, University of Texas at San Antonio, USA  {\itshape  email:
}\ttfamily changfeng.gui@utsa.edu} , Mingfeng Zhao  \thanks{Center for Applied Mathematics, Tianjin University, Tianjin, 300072, China,  {\itshape  email:
}\ttfamily mingfeng.zhao@tju.edu.cn}
}
\begin{document}
\title{Propagation acceleration  in  reaction diffusion equations with anomalous diffusions
}

\maketitle


\begin{abstract}
In this paper, we are interested in the properties of solution of the nonlocal equation 
$$\begin{cases}
	u_t+(-\Delta)^su=f(u),\quad t>0, \ x\in\mathbb{R}\\
	u(0,x)=u_0(x),\quad x\in\mathbb{R}
\end{cases}
$$
where $0\le u_0<1$ is a Heaviside type function, $\Delta^s$ stands for the fractional Laplacian with $s\in (0,1)$, and $f\in C([0,1],\R^+)$ is a non negative nonlinearity such that $f(0)=f(1)=0$ and $f'(1)<0$.  In this context, it is known that the solution $u(t,s)$ converges locally uniformly to 1 and our aim here is to understand how fast this invasion process occur. When $f$ is a Fisher-KPP type nonlinearity and $s \in (0,1)$, it is known that the level set of the solution $u(t,x)$ moves at an exponential speed whereas when $f$ is of ignition type and $s\in \left(\frac{1}{2},1\right)$ then the  level set of the solution  moves at a constant  speed. \\
In this article, for general monostable nonlinearities $f$ and any $s\in (0,1)$ we derive generic estimates on the position of the level sets of the solution $u(t,x)$ which then  enable us to describe more precisely the behaviour of this invasion process. In particular, we obtain a algebraic generic  upper bound on the "speed" of level set highlighting the  delicate  interplay of $s$ and $f$ in the existence of an exponential acceleration process.   
When $s\in\left (0,\frac{1}{2}\right]$ and $f$ is of ignition type, we also complete the known description of the behaviour of $u$ and give a precise asymptotic of the speed of the level set in this context. Notably,  we prove that the level sets accelerate when $s\in\left(0,\frac{1}{2}\right)$ and that in the critical case $s=\frac{1}{2}$ although no travelling front can exist, the level sets  still move asymptotically at a constant speed. These new results are in sharp contrast with the bistable situation where no such acceleration  may occur, hightligting therefore the qualitative difference between the two type of nonlinearities.    
\end{abstract}
\tableofcontents

\section{Introduction}

The study of propagation phenomena is a classical topic in  analysis as its provide a robust way to understand some pattern formations that arises in a wide range of context ranging from population dynamics in ecology \cite{Fisher1937,Kolmogorov1937}, to combustion \cite{Kanel1961} and phase transition \cite{Aronson1978}. Concretely, this often leads  to analyse the asymptotic properties of  the solution  $u(t,x)$ of the parabolic problem used to model the phenomenon considered.     
When this model is a reaction diffusion equation, this lead then to the study of the properties of the solutions of   
\begin{equation}\label{rd}
\begin{cases}
\partial_t u(t,x)=\Delta u(t,x) +f(u(t,x)) \quad \text{ for  } t>0, x\in \R^N\\
u(0,x)=u_0(x)
\end{cases}
\end{equation}
with respect to the nonlinearity $f$ and the initial data $u_0$. 
In this particular situation,   when $f$ is a smooth bistable, ignition or monostable nonlinearity, say $f$ Lipschitz such that $f(0)=f(1)=0$ , $f'(1)<0$, it is known that the solution of the equation \eqref{rd} can exhibit some phase transition behaviour. More precisely, for a Heaviside type initial datum $u_0$ i.e. $u_0(x)=\mathds{1}_{ H_e}(x)$ where $H_e$ denotes a half-space $\{x\in \R^N\,|\, x\cdot e<0\}$, then the solution $u(t,x)$ of\eqref{rd} converges locally uniformly as $t\to +\infty$ to $1 $ and the ``invasion process'' resulting of this initial datum can be characterised by  particular solutions of \eqref{rd}  called  planar front $\varphi(x\cdot e-ct)$ \cite{Aronson1978,Fife1977,Kolmogorov1937,Weinberger1982}, where $(\varphi, c)$ solves here the following equations
\begin{equation}
\begin{cases}
\varphi''(z)+c\varphi'(z)+f(\varphi(z))=0 \quad \text{ for  }  z\in \R,\\
\ds{\lim_{z\to-\infty }\varphi(z)=1}, \\
\ds{\lim_{z\to+\infty }\varphi(z)= 0.}
\end{cases}
\end{equation} 
In particular, for any $\lambda \in (0,1)$ the superlevel set  $E_{\lambda}(t):=\{x\in \R^N|u(t,x)\ge \lambda\}$ grows at a constant speed. That is there exists $x^+(\lambda),x^-(\lambda)$ in $\R^N$ and a family of Half-space $H^+(t)$ defined by 
$$
H^+(t):=\{x\in \R^N\,|\, x\cdot e -ct \le 0\}
$$
 such that $E_\lambda$ satisfies  
$$x^-(\lambda)+H^+(t)\subset E_\lambda (t)\subset  x^+(\lambda)+H^+(t).$$

Thanks to the comparison principle satisfied by such semi-linear equations \eqref{rd}, clearly this particular  phase transition behaviour appears also for other type of initial data $u_0\ge \mathds{1}_{H_e}$ that have some decay as $x\cdot e \to -\infty$. For those initial data, we may then wonder if the above description of the behaviour of superlevel set  $E_\lambda$ still holds true and if not how can we characterise it.  
As shown in \cite{Alfaro2017a,Hamel2010,Kay2001,Needham1999,Sherratt2005}, the above characterisation may not hold in general and  in some situation an accelerated transition may occur. 
Indeed when $N=1$  and for a monostable $f$ of KPP type, that is $f\in C([0,1])$ such that $f(0)=f(1)=0$, $f>0$, $f'(0)>0,f'(1)<0$, and such that $f(s)\le f'(0)s$,   then for $u_0(x)>0$   F. Hamel and L. Roques have obtained in \cite{Hamel2010} a sharp description of the speed  of the level line of the  solution of the corresponding Cauchy problem. In particular, they show that when $u_0$ is such that $u_0(x)\sim \frac{1}{x^\alpha} $, as $x\to +\infty$, then 
 the level lines of the  solution  move exponentially fast. That is, for any $\lambda \in (0,1)$ there exists  points $x(t)\in E_\lambda(t)$ such that $x(t)\sim e^{f'(0)t}$. More generally, they prove that
\begin{theorem}
 Let $u_0$ be for regular  ($C^2$) nonincreasing initial data $u_0$ on some semi-infinite interval
$[\xi_0,+\infty)$ and such that $$\partial_{xx}u_0 (x) = o(u_0 (x))\quad \text{ as }\quad x \to +\infty.$$
Then, for any $\lambda \in (0, 1)$, $\eps \in (0, f'(0))$,
$\mu > 0$ and $\nu > 0$, there exists $T_{\lambda,\eps,\mu,\nu}\ge t_\lambda$ such that
 $$ \Gamma_\lambda(t) \subset u_0^{-1}([ \mu e^{-(f'(0)+\eps)t}, \nu e^{-(f'(0)-\eps)t}]), $$
 where $\Gamma_\lambda$ denotes $$\Gamma_\lambda(t):=\{x\in \R\, |\, u(t,x)=\lambda\} $$
\end{theorem} 
 
 From this result, we can see the clear dependence  of the speed of the level sets of the solution $u(t,x)$ with respect to   the decay behaviour of $u_0$.
 Similar sharp descriptions of the speed of the level sets have been obtain for more general monostable type of nonlinearity, see for example \cite{Alfaro2017a,Kay2001,Needham1999,Sherratt2005}. On the other hand, thanks to the work of Fife and Mc Leod \cite{Fife1977}, and Alfaro \cite{Alfaro2017a} we can see that  accelerated transitions will never  occur when the non-linearity considered is bistable or of ignition type.


In this spirit, in this paper we are interested in propagation acceleration phenomena that  are caused  by  anomalous diffusions such as
super diffusions,   which plays important roles in various physical, chemical, biological and geological processes. (See, e.g.,
\cite{Volpert13} for a brief summary and references therein.)
A typical feature of such  anomalous diffusions is related to L\'evy  stochastic processes which may  possesses  discontinuous "jumps"  in their paths  and   have long range dispersal,  while the standard  diffusion is related to the Brownian motion.  Analytically,  certain   L\'evy processes ($\alpha$ stable) may  be modeled by their  infinitesimal generators which are  fractional Laplace operators $(-\Delta )^{s} u $ with $0<s<1$, whose Fourier transformation  $\widehat {(-\Delta
)^{s} u}  $ is $ (2\pi|\xi|)^{2s} \widehat{u}$. (See \cite{Metzler}.)

\noindent More precisely, we consider the following one-dimensional reaction-diffusion equation involving the fractional Laplacian:
\begin{eqnarray}\label{eqn:1-1}
	\left\{
	\begin{array}{l}
	u_t+(-\Delta)^su=f(u),\quad t>0, \ x\in\mathbb{R}\\
	u(0,x)=u_0(x),\quad x\in\mathbb{R}
	\end{array}
	\right.
\end{eqnarray}

\noindent where 
\begin{itemize}
	\item[(a)] $(-\Delta)^s$ ($0<s<1$) denotes the fractional Laplacian operator:
	\[ (-\Delta)^s u(x)=C_{1,s}\textnormal{P.V. }\int_{\mathbb{R}} \frac{u(x)-u(y)}{|x-y|^{1+2s}} dy,   \]
	
	\noindent where $C_{1,s}$ is a positive normalization constant in the sense that $\widehat{ (-\Delta)^s u}(\xi)=|\xi|^{2s}\widehat{u}(\xi)$. For simplicity, in the whole article, let's assume that $C_{1,s}=1$ after a suitable normalization.
	
	\item[(b)] $f$ is a nice function on $[0,1]$.
	
	\item[(c)] $u_0(x)$ is the initial condition.
	
\end{itemize}
The precise assumptions on $f$ and $u_0$ will be given later on.

Along with other types of nonlocal models (integrodifferential or integrodifference)  such nonlocal fractional  reaction diffusion model  \eqref{eqn:1-1} has received a lot of attention lately.  Contrary to the standard reaction diffusion equation \eqref{rd}, accelerated transitions can be observed for Heaviside type initial data \cite{Nec08, Mancinelli, Cabre2013, Engler, Gui2015} in the anomalous reaction diffusion systems. The mechanism that triggers the acceleration in this situation is then intrinsically  different  from that in the classical reaction diffusion and   seems governed  by  subtle interplay between the long range jumps in the diffusion processes and the strength of the pushes and pullings in the reaction part,  mathematically, i.e.,  the tails of the kernel and the properties of nonlinearity $f$ considered.   Namely, when $f$ is of bistable type then planar wave exists for all $s \in (0,1)$ \cite{Gui2015a, Palatucci2013} and the solution to \eqref{eqn:1-1} with a reasonable Heaviside  initial data $u_0$ will converge to a planar front, see \cite{Achleitner2015}. On the other hand, for the same initial data but for a $KPP$ type nonlinearity, the solution will accelerate exponentially fast \cite{Cabre2013, Engler}, that is, for $x(t)\in \Gamma(t)$ we have $x(t)\sim e^{f'(0)t}$.  

For more general monostable nonlinearities $f$, including those of ignition type,  the picture is less clear and only results  on the existence/ non-existence of planar front have been obtained. More precisely, when $f$ is an ignition nonlinearity then a planar front can only exist in the  range $s\in \left(\frac{1}{2},1\right)$ see \cite{Gui2015,Mellet2009}. Whereas for a general monostable nonlinearity $f$(i.e. $f(t)\sim t^p(1-t))$)  the existence of a planar front only occurs when $p>2$ and  in the range $s \in (\frac{p}{2(p-1)},1)$ see \cite{Gui2015}. In the later case, this suggest that  as in the KPP case, an accelerated transition will then occur for any $s \in (0,1)$ when $1<p<2$. A natural question is then, as in the KPP case,  does the level sets move with an exponential speed when $1<p<2$? 

One objective of this paper is to answer to this question and give a more detailed characterisation of the speed of the level set for general monostable nonlinearities $f$.



\subsection{Main Results}
Let us now describe more precisely the assumptions we made and state our main results. 
\begin{assumption}[Degenerate monostable nonlinearity]\label{ass:1-1}
	The nonlinearity $f: [0,1]\longrightarrow [0,\Vert f\Vert_\infty]$ is of class $C^1$, and is of the monostable type, in the sense that
    	\begin{eqnarray}\label{eqn:1-2}
    	f(0)=f(1)=0,\qquad f(u)>0,\quad\textnormal{for all $u\in(0,1)$}.
    	\end{eqnarray}
    	
   The steady state $0$ is degenerate, in the sense that, there exist some constants $r>0$ and $\beta>1$ such that
   \begin{eqnarray}\label{eqn:1-3}
   	f(u)\leq ru^\beta,\quad\textnormal{for all $u\in[0,1]$}.
   \end{eqnarray}
    	
    The steady state $1$ is stable, in the sense that
    \begin{eqnarray}\label{eqn:1-4}
    		f'(1)<0.
    \end{eqnarray}

\end{assumption}

\begin{assumption}[Front like initial datum]\label{ass:1-2}
	The initial data $u_0:\mathbb{R}\longrightarrow[0,1]$ is of class $C^1$ and satisfies
    \begin{itemize}
    	\item[(a)] $0\leq u_0(x)\leq1$ for all $x\in\mathbb{R}$.
    	
    	\item[(b)] $\displaystyle \liminf_{x\rightarrow-\infty} u_0(x)>0$.
    	
    	\item[(c)] $u_0(x)\equiv0$ on $[a,+\infty)$ for some $a$.
    \end{itemize}
	
\end{assumption}

Under this two assumptions, we first prove that
\begin{theorem}\label{thm:1-1}
	For any $0<s<1$, assume that the nonlinearity $f$ satisfies {\bf Assumption \ref{ass:1-1}}, and the initial data $u_0(x)$ satisfies {\bf Assumption \ref{ass:1-2}}. Let $u(t,x)$ be the solution to the problem \eqref{eqn:1-1} with the initial data $u_0(x)$, consider the superlevel set $E_\lambda(t)=\{ x\in\mathbb{R}| u(t,x)>\lambda  \}$ of the solution $u(t,x)$, and define
	\[ x_\lambda(t)=\sup E_\lambda(t).  \]
	
	\noindent If further assume that $\frac{\beta}{2s(\beta-1)}>1$, then for any $\lambda\in(0,1)$, there exist some constants $T_\lambda>0$ and $C(\lambda)>0$ such 
	\begin{eqnarray*}
		E_\lambda(t)\subseteq (-\infty, x_\lambda(t)),\quad \textnormal{and}\quad 	x_\lambda(t)\leq C(\lambda) t^{\frac{\beta}{2s(\beta-1)}  },\qquad \forall t>T_\lambda.
	\end{eqnarray*}

\end{theorem}

When $\beta>2$ and $\frac{\beta}{2s(\beta-1)}\leq1$, the existence of the traveling wave to the problem \eqref{eqn:1-1} provided $\frac{\beta}{2s(\beta-1)}\leq1$  was proved by Gui and Huan in \cite{Gui2015} meaning that for the solution $u(t,x)$ to \eqref{eqn:1-1} with some front-like data, if we look at the level set of $u(t,x)$, then the spatial variable $x$ may linearly depend on the time variable $t$. In this sense, our condition $\frac{\beta}{2s(\beta-1)}>1$ is sharp. In addition, we can observe from our results that when $\beta>1$ then the  level set of the solution  $u(t,x)$ to the equation \eqref{eqn:1-1}  moves at most at a polynomial rate i.e $\ds{x_\lambda(t)\sim t^\gamma}$ with $\ds{\gamma:=\sup\left\{1;\frac{1}{2s}+\frac{1}{\beta -1}\right\}}$.  These results  are in sharp contrasts with the results of Cabre et al. \cite{Cabre2013} for the KPP case. In particular, they highlight the fact  that the exponential acceleration is strongly thigh to the non-degeneracy of the nonlinearity  $f$ and only occur when $f$ is such that $f'(0)>0$, a situation that allows an exponential growth at low density.    
	

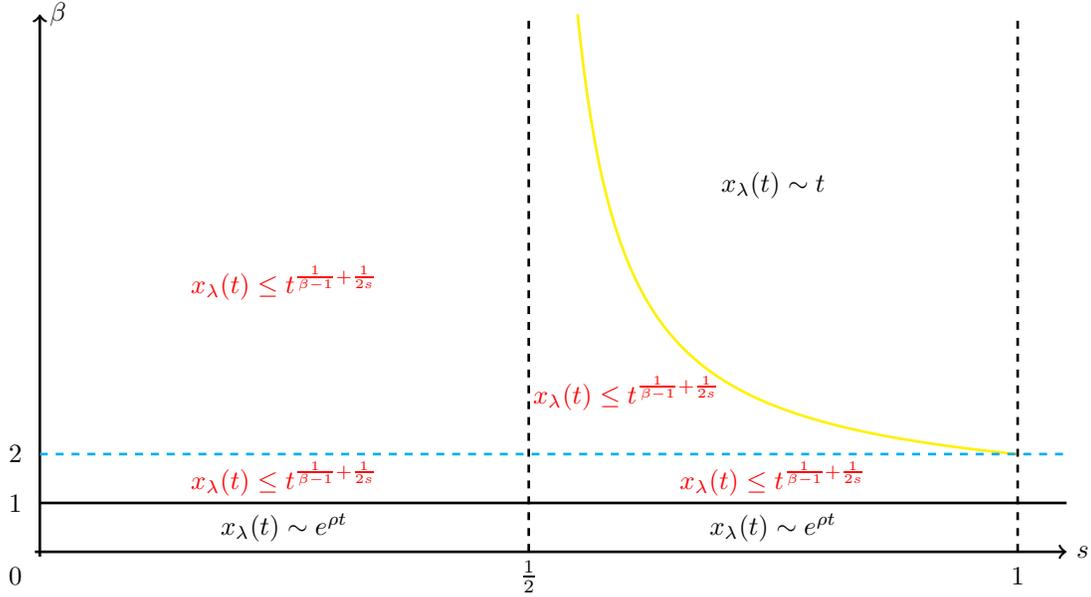
\begin{center}
	\begin{tikzpicture}[scale=0.65]
	\draw[line width=1pt,color=black,->] (-0.1,0)--(21,0) node[right]{$s$};
	
	\draw[line width=1pt,,color=black,->] (0,-0.1)--(0,11) node[right]{$\beta$};
	
	\draw[line width=1pt, color=yellow, domain=11:20, samples=1000,variable=\t] plot ({\t},{(-2*(\t/20)/(1-2*(\t/20))) });
	
	\draw[line width=1pt] (0,1)--(21,1) node[right]{};
	\draw[line width=1pt, color=cyan, style=dashed] (0,2)--(21,2) node[right]{};
	\draw[line width=1pt,style=dashed] (20,0)--(20,11);
	\draw[line width=1pt,style=dashed] (10,0)--(10,11);
	
	\node at (-0.5,-0.5){$0$};
	
	\node at (-0.5,1){$1$};
	\node at (-0.5,2){$2$};
	
	\node at (20,-0.5){$1$};
	\node at (10,-0.5){$\frac{1}{2}$};

	\node at (5,0.5){$x_\lambda(t) \sim e^{\rho t}$};
	\node at (5,1.5){\textcolor{red}{$x_\lambda(t) \le t^{\frac{1}{\beta-1}+\frac{1}{2s}}$}};
	\node at (5,5.5){\textcolor{red}{$x_\lambda(t) \le t^{\frac{1}{\beta-1}+\frac{1}{2s}}$}};
	
    \node at (15,0.5){$x_\lambda(t) \sim e^{\rho t}$};	
	\node at (15,1.5){\textcolor{red}{$x_\lambda(t) \le t^{\frac{1}{\beta-1}+\frac{1}{2s}}$}};
	\node at (12,3.25){\textcolor{red}{$x_\lambda(t) \le t^{\frac{1}{\beta-1}+\frac{1}{2s}}$}};
	\node at (15,7.5){$x_\lambda(t) \sim  t$};
	\end{tikzpicture}
\end{center}

Next, we prove a first lower bound of the speed of the level set. Namely, we show that
\begin{theorem}[A rough lower bound]\label{thm:1-2}
	For any $0<s<1$, assume that the nonlinearity $f$ satisfies $f(u)\geq0$ for all $u\in[0,1]$, and the initial data $u_0(x)$ satisfies {\bf Assumption \ref{ass:1-2}}. Let $u(t,x)$ be the solution to the problem \eqref{eqn:1-1} with the initial data $u_0(x)$, consider the superlevel set $E_\lambda(t)=\{ x\in\mathbb{R}| u(t,x)>\lambda  \}$ of the solution $u(t,x)$, and define
	\[ x_\lambda(t)=\sup E_\lambda(t).  \]
	
	 \noindent Then for any $\lambda\in(0,1)$, there exists some constants $T_\lambda'>0$ and $C'(\lambda)>0$ such 
	\begin{eqnarray*}
		x_\lambda(t)\geq C'(\lambda) t^{\frac{1}{2s}  },\qquad\forall t>T_\lambda'
	\end{eqnarray*}

\end{theorem}

  Combining the later with the upper bound obtained in  Theorem \ref{thm:1-2},  as a immediate corollary we then get      

\begin{corollary}\label{cor:1-1}
    For any $0<s<1$, assume that the nonlinearity $f$ satisfies {\bf Assumption \ref{ass:1-1}} and $f(u)\geq0$ for all $u\in[0,1]$, and the initial data $u_0(x)$ satisfies {\bf Assumption \ref{ass:1-2}}. Let $u(t,x)$ be the solution to the problem \eqref{eqn:1-1} with the initial data $u_0(x)$, consider the superlevel set $E_\lambda(t)=\{ x\in\mathbb{R}| u(t,x)>\lambda  \}$ of the solution $u(t,x)$, and define
		\[ x_\lambda(t)=\sup  E_\lambda(t).  \]
		
    \noindent If further assume that $\frac{\beta}{2s(\beta-1)}>1$, then for any $\lambda\in(0,1)$, there exists some constants $T_\lambda>0$, $C(\lambda)>0$ and $C'(\lambda)>0$ such 
		\begin{eqnarray*}
			C'(\lambda) t^{\frac{1}{2s}}\leq 	x_\lambda(t)\leq C(\lambda) t^{\frac{\beta}{2s(\beta-1)}  },\qquad \forall t>T_\lambda.
		\end{eqnarray*}

\end{corollary}

Although these first estimates on the speed seem rather crude this are still quite interesting, in particular in the case $0<s<\frac{1}{2}$, as they give a very  simple way of showing the non-existence of the traveling wave solution to the problem \eqref{eqn:1-1} with any general non negative function $f$ and in particular for the Fisher-KPP nonlinearity. These results also highlight a fundamental difference between nonlocal model versus local model when considering an ignition type nonlinearity. Indeed, when the nonlinearity $f$ is of  ignition type,  we  can easily  deduce from the work of  Alfaro \cite{Alfaro2017a} that accelerated transitions never occur  in the classical reaction diffusion model \eqref{rd} whereas  they do in the nonlocal reaction diffusion \eqref{eqn:1-1} when $s\in(0,1/2)$. This is also a clear evidence that in the nonlocal setting, unlike in the local setting (\eqref{rd}) the two types of nonlinearities: bistable and ignition type are not alike in the sense that the dynamic obtained are completely different. In this nonlocal setting,  a  condition on the decay of the tail of the kernel appears then of crucial importance  in order to guarantee the existence of traveling front. Namely, from our results we can see that when $f$ is non negative the kernel must satisfy some first moment integrability condition to expect to observe traveling front solutions.
This finite first moment condition suggests that a similar  condition should hold true  as well for convolution type nonlocal models studied in \cite{Coville2007d}  as the these two models shares many similarities. That is,  in such convolution type models, for a traveling front to exist the kernel need to satisfy a first moment condition.

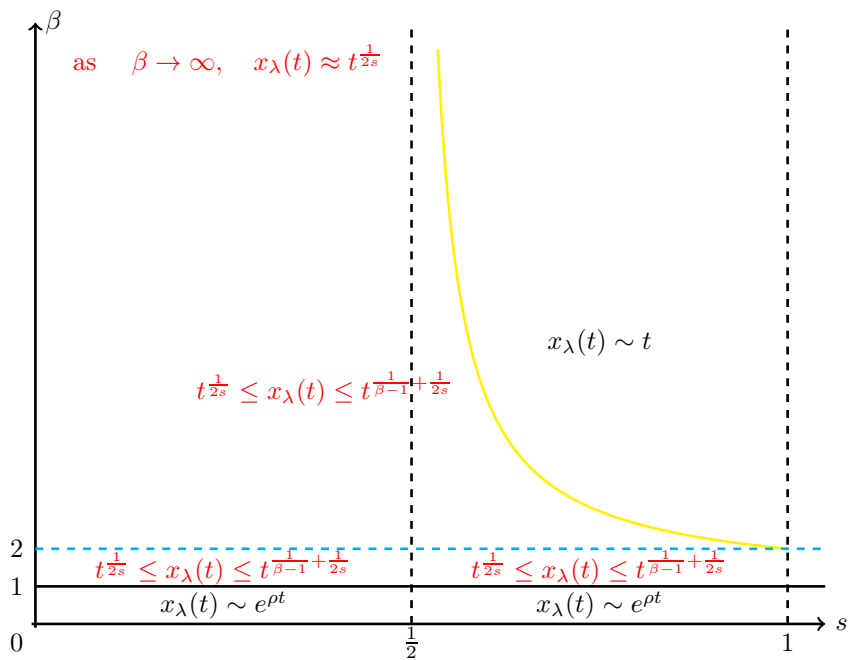
\begin{center}
	\begin{tikzpicture}[scale=0.5]
	\draw[line width=1pt,color=black,->] (-0.1,0)--(21,0) node[right]{$s$};
	
	\draw[line width=1pt,,color=black,->] (0,-0.1)--(0,16) node[right]{$\beta$};
	
	\draw[line width=1pt, color=yellow, domain=10.7:20, samples=1000,variable=\t] plot ({\t},{(-2*(\t/20)/(1-2*(\t/20))) });
	
	\draw[line width=1pt] (0,1)--(21,1) node[right]{};
	\draw[line width=1pt, color=cyan, style=dashed] (0,2)--(21,2) node[right]{};
	\draw[line width=1pt,style=dashed] (20,0)--(20,16);
	\draw[line width=1pt,style=dashed] (10,0)--(10,16);
	
	\node at (-0.5,-0.5){$0$};
	
	\node at (-0.5,1){$1$};
	\node at (-0.5,2){$2$};
	
	\node at (20,-0.5){$1$};
	\node at (10,-0.5){$\frac{1}{2}$};

	\node at (5,0.5){$x_\lambda(t) \sim e^{\rho t}$};
	\node at (5,1.5){\textcolor{red}{$t^{\frac{1}{2s}}\le x_\lambda(t) \le t^{\frac{1}{\beta-1}+\frac{1}{2s}}$}};
	\node at (5,15){\textcolor{red}{$\text{ as }\quad \beta \to \infty, \quad x_\lambda(t)\approx t^{\frac{1}{2s}}$}};
	
    \node at (15,0.5){$x_\lambda(t) \sim e^{\rho t}$};	
	\node at (15,1.5){\textcolor{red}{$t^{\frac{1}{2s}}\le x_\lambda(t) \le t^{\frac{1}{\beta-1}+\frac{1}{2s}}$}};
	\node at (7.7,6.25){\textcolor{red}{$t^{\frac{1}{2s}}\le x_\lambda(t) \le t^{\frac{1}{\beta-1}+\frac{1}{2s}}$}};
	\node at (15,7.5){$x_\lambda(t) \sim  t$};
	\end{tikzpicture}
\end{center}

Let us look now more deeply at the consequences of these first estimates on the speed for the combustion model and for  supercritical fractional Laplacians (that is, $0<s<\frac{1}{2}$). 
In this situation, from the above estimate  we can in fact derive a sharp estimate on the speed of propagation. Namely, we show

\begin{corollary}[Combustion model for supercritical and critical fractional Laplacians]\label{cor:1-2}
	For any $0<s\leq \frac{1}{2}$, assume that the initial data $u_0(x)$ satisfies {\bf Assumption \ref{ass:1-2}}, and the nonlinearity $f$ is a combustion type nonlinearity, in the sense that there exists some $\theta\in(0,1)$ such that
	\begin{eqnarray}\label{eqn:1-5}
		f(1)=0=f(u),\quad\textnormal{for all $u\in[0,\theta]$},\qquad\textnormal{and}\qquad f(u)>0\quad\textnormal{for all $u\in(\theta,1)$}
	\end{eqnarray} 
	
	\noindent Let $u(t,x)$ be the solution to the problem \eqref{eqn:1-1} with the initial data $u_0(x)$, consider the superlevel set $E_\lambda(t)=\{ x\in\mathbb{R}| u(t,x)>\lambda  \}$ of the solution $u(t,x)$, and define
	\[ x_\lambda(t)=\sup E_\lambda(t).  \]
	
%
	\noindent Then for any $\eps>0$, and for any $\lambda\in(0,1)$, there exists some constants $T_{\lambda,\eps}>0$, $C(\lambda,\eps)>0$ and $C'(\lambda)>0$ such 
	\begin{eqnarray*}
		C'(\lambda) t^{\frac{1}{2s}}\leq 	x_\lambda(t)\leq C(\lambda,\eps) t^{\frac{1}{2s}+\eps  },\qquad \forall t>T_{\lambda,\eps}.
	\end{eqnarray*}
	
\end{corollary}

The proof of this corollary is quite straightforward. Indeed,  the combustion model can be thought as some limit case of the degenerated monostable situation (i.e., $f(u)$ monostable with $f^{(k)}(0)=0$ for all $k\in \N$). In particular,   for any combustion nonlinearity $f$ we may find a constant $C_0>0$ such that for all $ \beta>1$ we have $$ f(u)\le f_\beta(u):= C_0u^\beta(1-u).$$
Recall that since we assume  that the fractional Laplacian is  either super-critical or critical (i.e.   $s\in (0,\frac{1}{2}]$) then we can check that for all $ \beta>1$ the condition below is satisfied
$$\frac{\beta}{2s(\beta -1)}=\frac{1}{2s}+\frac{1}{2s(\beta-1)}>1 $$ and then 
using a standard comparison argument and  {\bf Corollary \ref{cor:1-1}} we may deduce that for any $\beta>1$ there exists $C(\beta)$ and $T_\beta$ such that for all $t\ge T_\beta$
$$ x_\lambda(t)\le C(\lambda,\beta)t^{\frac{1}{2s}+\frac{1}{2s(\beta-1)}}$$
The results of  {\bf Corollary \ref{cor:1-2}}  follows then by picking $\beta$ so large that  we have $\frac{1}{2s(\beta-1)}\le \eps$.

	Note that  this  estimate is sharp in the sense it gives the right asymptotic for the speed of the level set i.e. we  get $x_\lambda(t)\sim t^{\frac{1}{2s}}$ as $t\to \infty$. It also provides a useful information for the critical case $s=\frac{1}{2}$, where we see that the level set moves asymptotically with a constant speed although there is no  existence of a traveling front in this situation.

	
Lastly, in the spirit of \cite{Alfaro2017}, let us obtain a finer lower bound on the speed for general degenerate monostable nonlinearities $f$, i.e. $\exists \, \beta\in (1,+\infty), \text{ such that } \lim_{u\to 0}\frac{f(u)}{u^\beta}=l>0$.  

\begin{theorem}[A finer lower bound]\label{thm:1-3}
	For any $0<s<1$, assume that the nonlinearity $f$ satisfies {\bf Assumption \ref{ass:1-1}} and $f(u)\geq r_1u^\beta$ as $u\rightarrow 0^+$ for some small $r_1>0$, and the initial data $u_0(x)$ satisfies {\bf Assumption \ref{ass:1-2}}. Let $u(t,x)$ be the solution to the problem \eqref{eqn:1-1} with the initial data $u_0(x)$, consider the superlevel set $E_\lambda(t)=\{ x\in\mathbb{R}| u(t,x)>\lambda  \}$ of the solution $u(t,x)$, and define
	\[ x_\lambda(t)=\sup  E_\lambda(t).  \]
	
	\noindent If further assume that $\frac{1}{2s(\beta-1)}>1$, then for any $\lambda\in(0,1)$, there exists some constants $T_\lambda>0$, $C(\lambda)>0$ and $C'(\lambda)>0$ such 
	\begin{eqnarray*}
		C'(\lambda) t^{\frac{1}{2s(\beta-1)}}\leq 	x_\lambda(t)\leq C(\lambda) t^{\frac{\beta}{2s(\beta-1)}  },\qquad \forall t>T_\lambda.
	\end{eqnarray*}

\end{theorem}

	Notice that $\frac{1}{2s}<\frac{1}{2s(\beta-1)}$ if and only if $1<\beta<2$. Hence when $\frac{1}{2}<s<1$ and $1<\beta<2$, the lower bound in {\bf  Theorem \ref{thm:1-3}} is better than the one in {\bf Theorem \ref{thm:1-2}}. From these estimates we can then deduce the following generic estimate:

	\begin{theorem}[generic bound]\label{thm:1-4}
	For any $0<s<1$, assume that the nonlinearity $f$ satisfies {\bf Assumption \ref{ass:1-1}} and $f(u)\geq r_1u^\beta$ as $u\rightarrow 0^+$ for some small $r_1>0$, and the initial data $u_0(x)$ satisfies {\bf Assumption \ref{ass:1-2}}. Let $u(t,x)$ be the solution to the problem \eqref{eqn:1-1} with the initial data $u_0(x)$, consider the superlevel set $E_\lambda(t)=\{ x\in\mathbb{R}| u(t,x)>\lambda  \}$ of the solution $u(t,x)$, and define
	\[ x_\lambda(t)=\sup  E_\lambda(t).  \]
	
	\noindent Then for any $\lambda\in(0,1)$, there exists some constants $T_\lambda>0$, $C(\lambda)>0$ and $C'(\lambda)>0$ such 
	\begin{eqnarray*}
		C'(\lambda) t^{\sup\left\{\frac{1}{2s(\beta-1)};\frac{1}{2s}\right\}}\leq 	x_\lambda(t)\leq C(\lambda) t^{\frac{1}{2s(\beta-1)} +\frac{1}{2s} },\qquad \forall t>T_\lambda.
	\end{eqnarray*}

\end{theorem}

This last results clearly indicate that the speed of the level sets is the result of a fine interplay between the diffusion process intimately linked  to the quantity $t^1/2s$ and the reaction term $f$ which, as we will see in the proof, is strongly linked to the quantity $t^{\frac{1}{2s(\beta-1)}}$.

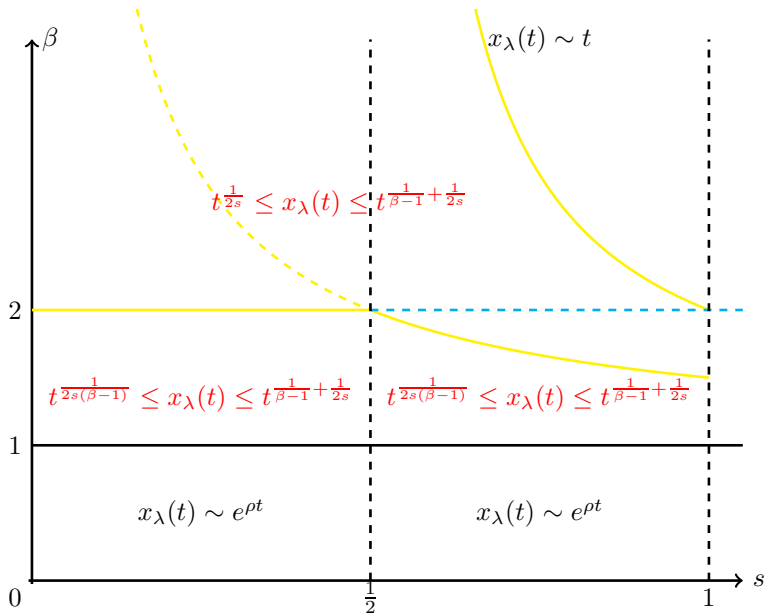
\begin{center}
	\begin{tikzpicture}[scale=0.45]
	\draw[line width=1pt,color=black,->] (-0.1,0)--(21,0) node[right]{$s$};
	
	\draw[line width=1pt,,color=black,->] (0,-0.1)--(0,16) node[right]{$\beta$};
	
	\draw[line width=1pt, color=yellow, domain=13.1:20, samples=1000,variable=\t] plot ({\t},{4*(1- (1/(1-2*(\t/20)))) });
	\draw[line width=1pt, color=yellow, style=dashed, domain=3.1:10, samples=1000,variable=\t] plot ({\t},{4*(1+1/(2*(\t/20))) });
	\draw[line width=1pt, color=yellow, domain=10:20, samples=1000,variable=\t] plot ({\t},{4*(1+1/(2*(\t/20))) });
	\draw[line width=1pt] (0,4)--(21,4) node[right]{};
	\draw[line width=1pt, color=yellow] (0,8)--(10,8) node[right]{};
	\draw[line width=1pt, color=cyan, style=dashed] (10,8)--(21,8) node[right]{};
	\draw[line width=1pt,style=dashed] (20,0)--(20,16);
	\draw[line width=1pt,style=dashed] (10,0)--(10,16);
	
	\node at (-0.5,-0.5){$0$};
	
	\node at (-0.5,4){$1$};
	\node at (-0.5,8){$2$};
	
	\node at (20,-0.5){$1$};
	\node at (10,-0.5){$\frac{1}{2}$};

	\node at (5,2){$x_\lambda(t) \sim e^{\rho t}$};
	\node at (5,5.5){\textcolor{red}{$t^{\frac{1}{2s(\beta-1)}}\le x_\lambda(t) \le t^{\frac{1}{\beta-1}+\frac{1}{2s}}$}};
	
    \node at (15,2){$x_\lambda(t) \sim e^{\rho t}$};	
	\node at (15,5.5){\textcolor{red}{$t^{\frac{1}{2s(\beta -1)}}\le x_\lambda(t) \le t^{\frac{1}{\beta-1}+\frac{1}{2s}}$}};
	\node at (9.1,11.25){\textcolor{red}{$t^{\frac{1}{2s}}\le x_\lambda(t) \le t^{\frac{1}{\beta-1}+\frac{1}{2s}}$}};
	\node at (15,16){$x_\lambda(t) \sim  t$};
	\end{tikzpicture}
\end{center}

\subsection{Further comments}
Before going to the proofs of our results, we would like to make some further comments.
First, we would like to emphasize that similar results were previously obtained  in \cite{Alfaro2017}  in the context of integrodifferential equation 
\begin{equation}\label{noloc}
\begin{cases}
\partial_tu(t,x)=J\star u(t,x) -u(t,x) + f(u(t,x)) \quad \text{ for }\quad t>0, x\in \R\\
u(0,x)=u_0(x)
\end{cases}
\end{equation}
where $J\star u$ stands for the standard  convolution and $J$ is a positive probability density with a finite first moment i.e $J\in L^1(\R)$ such that $J\ge 0,\int_{\R}J(z)\,dz=1,\int_{\R}J(z)|z|\,dz<+\infty$. The two equations \eqref{noloc}
 and \eqref{eqn:1-1} shares some similarities,  and in particular the equation \eqref{eqn:1-1} may be viewed as a reformulation of the equation \eqref{noloc}  but with a non integrable  singular kernel. However,  the treatment of the singularity is of crucial importance here and induces some tricky technical difficulties,  which the  ideas developed to analyse \eqref{noloc} seem not able to overcome.   Indeed, the challenge of singularity here is  intrinsic and related to the physical nature of the fractional  Laplacian.  The approach  here   is  hence  not just an adaptations of the proofs given in \cite{Alfaro2017},  and  we have to deal with the singularity carefully.   In particular, we go a step further in our understanding of the mechanism triggering acceleration by  describing the situation for $s \in (0,\frac{1}{2})$,  a situation which is not treated in \cite{Alfaro2017}
 at all.  We believe that some of the techniques developed here will  be  also useful to apprehend propagation phenomena in the equation \eqref{noloc} for kernels that do not satisfy  this first moment condition. In particular, the analysis presented here should provide the ground for a deeper understanding of nonlocal combustion problems modeled by the equation \eqref{noloc} studied in \cite{Coville2007d},    by ensuring that the  existence of traveling front is conditioned to a first moment property satisfied by the kernel.  Works in this direction are currently underway.

We also want to stress  that although our results gives some good insights on the speed of the level sets, apart from situations involving combustion nonlinearities  where a precise asymptotic is known,  there is  still a gap in our estimates and the right behaviour,that we believe is $t^{\frac{\beta}{2s(\beta-1)}}$ has not been capture yet.  Using new approaches, recent progress have been made  on the understanding of  acceleration phenomena in various situations , namely for  semilinear equation like \eqref{rd}  with a  nonlinear diffusion instead of the classical diffusion \cite{Alfaro2017b,Alfaro2019,Garnier2017}  as well as for the equation \eqref{noloc} with Fisher -KPP type nonlinearity \cite{Bouin2018,Garnier2011}.  
The different approaches developed in these works may be of some help in this task.  Works in this direction are also under consideration.

The paper is organised as follows.  In Section \ref{sec1}, we  prove  Theorem \ref{thm:1-1} and obtain the upper bound on the speed of the level set.Then in Section \ref{sec-l}, we obtain the generic lower bound on this speed, Theorem \ref{thm:1-2}. Finally, in the last section, Section \ref{sec3}, we prove the a refine estimate of this speed  when  a  degenerate monostable nonlinearity $f$ is considered, Theorem \ref{thm:1-3}.  
\section{Upper bound on the speed of the super level sets}\label{sec1}

\noindent \emph{Construction of a supersolution}: For some constant $p>0$ which will be determined later, let's define
\begin{eqnarray*}
	v_0(x)=\left\{
	\begin{array}{ll}
		1, &\textnormal{if $x\leq1$},\\
		\displaystyle \frac{1}{x^p}, &\textnormal{if $x>1$}.
	\end{array}
	\right.
\end{eqnarray*}

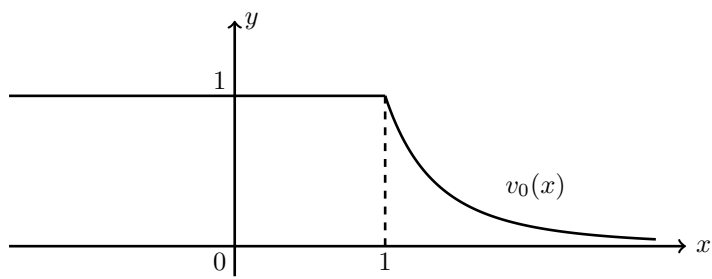
\begin{center}
	\begin{tikzpicture}[scale=2]
	\draw[line width=1pt,color=black,->] (-1.5,0)--(3,0) node[right]{$x$};
	
	\draw[line width=1pt,,color=black,->] (0,-0.2)--(0,1.5) node[right]{$y$};
	
	\draw[line width=1pt, color=black, domain=1:2.8, samples=100,variable=\t] plot ({\t},{1/(\t*\t*\t  ) });
	
	\draw[line width=1pt] (-1.5,1)--(1,1) node[right]{};
	
	\draw[line width=1pt,style=dashed] (1,0)--(1,1);
	
	\node at (-0.1,-0.1){$0$};
	
	\node at (-0.1,1.1){$1$};
	
	\node at (1,-0.1){$1$};
	
	\node at (2,0.4){$v_0(x)$};
	
	\end{tikzpicture}
\end{center}

For any $\gamma>0$, let $w(t,x)$ be the solution to the following initial-value problem:
\begin{eqnarray*}
	\left\{
	\begin{array}{l}
		\displaystyle \frac{dw(t,x)}{dt}=\gamma[w(t,x)]^\beta,\\
		w(0,x)=v_0(x)
	\end{array}
	\right.
\end{eqnarray*}

Since $\beta>1$, it's easy to solve the above problem and obtain
\begin{eqnarray*}
	w(t,x)&=&\frac{1}{\left[[v_0(x)]^{1-\beta}-\gamma(\beta-1)t  \right]^{\frac{1}{\beta-1}}}.  
\end{eqnarray*}

By the definition of $v_0(x)$, it's easy to see that $w(t,x)$ is well defined for $t\in\left[0,\frac{1}{\gamma(\beta-1)}\right)$ if $x\leq1$; and $w(t,x)$ is well defined for $t\in\left[0, \frac{x^{p(\beta-1)}}{\gamma(\beta-1)}\right)$ if $x>1$. When $t$ is fixed, the function $w(t,x)$ is decreasing with respect to $x$.

Let $x_0(t)=\left[ 1+\gamma(\beta-1)t  \right]^{\frac{1}{p(\beta-1)}}$ for all $t>0$, it's easy to know that $x_0(t)>1$, $w(t,x_0(t))=1$, $w(t,x)>1$ for all $x<x_0(t)$, and $w(t,x)<1$ for all $x>x_0(t)$. Let's consider the function 
\begin{eqnarray*}
	m(t,x)=\left\{
	\begin{array}{ll}
		1, &\textnormal{if $x\leq x_0(t)$},\\
		w(t,x),& \textnormal{if $x>x_0(t)$}.
	\end{array}
	\right.
\end{eqnarray*}

It's easy to see that $m(t,x)$ is well defined for all $t\geq0$ and all $x\in\mathbb{R}$, and $0<m(t,x)\leq 1$ for all $t\geq0$ and all $x\in\mathbb{R}$.

\begin{claimd}\label{claim:2-1}
	If $p+1\geq p\beta$, then there exists some constant $C(p,\beta)>0$ such that
	\[ |\partial_x m(t,x)|+|\partial^2_{xx} m(t,x)|\leq C(p,\beta),\qquad\forall t>0,\ \forall x\in\mathbb{R}.  \]	
	
\end{claimd}

\begin{proof}
	It's easy to see that $\partial_xm(t,x)=\partial^2_{xx}m(t,x)=0$ for all $(t,x)$ such that $x\leq x_0(t)$. Now for any $x>x_0(t)$, then
	\begin{eqnarray*}
		m(t,x)=w(t,x)=[ [v_0(x)]^{1-\beta}-\gamma(\beta-1)t  ]^{\frac{1}{1-\beta} }
	\end{eqnarray*}
	
	Since $x_0(t)>1$, then $v_0(x)=\frac{1}{x^p}$. A direct computation shows that
	\begin{eqnarray*}
		\partial_xm(t,x)&=&\partial_xw(t,x)\\
		&=&\frac{1}{1-\beta}\cdot [ [v_0(x)]^{1-\beta}-\gamma(\beta-1)t  ]^{\frac{1}{1-\beta} -1}\cdot (1-\beta)[v_0(x)]^{-\beta}\cdot v_0'(x)\\
		&=&-p[m(t,x)]^{\beta}\cdot x^{p\beta-p-1}.
	\end{eqnarray*}

    Since $\beta>1$, $p+1\geq p\beta$, $0<m(t,x)\leq 1$, and $x>x_0(t)>1$, then we have
    \[  |\partial_xm(t,x)|\leq p.  \]
    
    On the other hand, we have 
    \begin{eqnarray*}
    	\partial^2_{xx}m(t,x)&=&-p\partial_x\left[ [m(t,x)]^{\beta}\cdot x^{p\beta-p-1}  \right]\\
    	&=&-p\left[\beta[m(t,x)]^{\beta-1}\cdot \partial_xw(t,x)\cdot x^{p\beta-p-1}+[m(t,x)]^\beta\cdot (p\beta-p-1)\cdot x^{p\beta-p-2}    \right]\\
    	&=&-p\left[\beta[m(t,x)]^{\beta-1}\cdot \left[ -p[m(t,x)]^{\beta}\cdot x^{p\beta-p-1} \right]\cdot x^{p\beta-p-1}+[m(t,x)]^\beta\cdot (p\beta-p-1)\cdot x^{p\beta-p-2}    \right]\\
    	&=&p^2\beta[m(t,x)]^{2\beta-1}x^{2(p\beta-p-1)}+p\beta(p+1-p\beta)[m(t,x)]^{\beta}x^{p\beta-p-2}.
    \end{eqnarray*}

   Since $\beta>1$, $p+2>p+1\geq p\beta$, $0<m(t,x)\leq 1$, and $x>x_0(t)>1$, then we have
   \[  |\partial^2_{xx}m(t,x)|\leq p^2\beta+p\beta(p+1-p\beta).  \]

\end{proof}

\begin{claimd}\label{claim:2-2}
	If $p+1\geq p\beta$, then there exists some constant $C(s,p,\beta)>1$ such that
	\[  | (-\Delta)^sm(t,x)|\leq C(s,p,\beta),\qquad\forall t>0,\ \forall x\in\mathbb{R}.  \]
	
\end{claimd}

\begin{proof}
	By the definition of $(-\Delta)^sm(t,x)$, it's easy to see that
	\begin{eqnarray*}
		(-\Delta)^sm(t,x)&=&\frac{1}{2}\int_{\mathbb{R}} \frac{ m(t,x+h)+m(t,x-h)-2m(t,x)  }{|h|^{1+2s}} dh\\
		&=&\frac{1}{2}\int_{|h|\geq1 } \frac{ m(t,x+h)+m(t,x-h)-2m(t,x)  }{|h|^{1+2s}} dh+\frac{1}{2}\int_{|h|<1} \frac{ m(t,x+h)+m(t,x-h)-2m(t,x)  }{|h|^{1+2s}} dh
	\end{eqnarray*}
	
	Since $0<m(t',x')\leq1$ for all $t'>0$ and all $x'\in\mathbb{R}$, by {\bf Claim \ref{claim:2-1}}, then there exists some $C_1(s,p,\beta)>0$ such that
	\begin{eqnarray*}
		|(-\Delta)^sm(t,x)|&\leq& \frac{1}{2}\int_{|h|\geq1} \frac{4}{|h|^{1+2s}} dh+\frac{1}{2}\int_{|h|<1} \frac{ C_1|h|^{2}   }{|h|^{1+2s}} dh\\
		&=&C(s,p,\beta).
	\end{eqnarray*}
	
\end{proof}

\begin{claimd}\label{claim:2-3}
	For any $(t,x)$ such that $x\leq x_0(t)$, we have
	\[ \partial_tm(t,x)+(-\Delta)^sm(t,x)-f(m(t,x))>0.  \]
	
\end{claimd}

\begin{proof}
	In fact, since $x\leq x_0(t)$, then $m(t,x)=1$. By {\bf Assumption \ref{ass:1-1}}, then   $f(m(t,x))=f(1)=0$. By the definition of $m(t',x')$, it's easy to see that $\partial_tm(t,x)=0$. Since $m(t,x)=1$, $0<m(t',x')\leq1$ for all $t'>0$ and all $x'\in\mathbb{R}$, by the definition of $(-\Delta)^sm(t,x)$, it's easy to obtain that $(-\Delta)^sm(t,x)>0$. In summary, we obtain that 
	\[ \partial_tm(t,x)+(-\Delta)^sm(t,x)-f(m(t,x))>0.  \]
	
\end{proof}

\begin{claimd}\label{claim:2-4}
	If $p+1\geq p\beta$, let $C_1$ be the positive constant defined in {\bf Claim \ref{claim:2-2}}, for any fixed $\gamma>\gamma_0:=r+2C_1$, let's define
	\[ x_\gamma(t)=\left[\left(\frac{\gamma-r}{C_1}\right)^{\frac{\beta-1}{\beta} }+\gamma(\beta-1)t   \right]^{\frac{1}{p(\beta-1)}}.  \]
	
	\noindent Then $x_0(t)<x_\gamma(t)$ for all $t>0$, and for any $(t,x)$ such that $x_0(t)<x\leq x_\gamma(t)$, we have
	\[ \partial_tm(t,x)+(-\Delta)^sm(t,x)-f(m(t,x))\geq0.  \]
	
\end{claimd}

\begin{proof}
	Since $\gamma-r\geq 2C_1>2$, since $\beta>1$, then $x_\gamma(t)>[1+\gamma(\beta-1)t]^{\frac{1}{p(\beta-1)}}=x_0(t)$. Now for any $(t,x)$ such that $x_0(t)<x\leq x_\gamma(t)$, by {\bf Assumption \ref{ass:1-1}} and {\bf Claim \ref{claim:2-2}}, then we have
	\begin{eqnarray*}
		\partial_tm(t,x)+(-\Delta)^sm(t,x)-f(m(t,x))&=&\partial_tw(t,x)+(-\Delta)^sm(t,x)-f(w(t,x))\\
		&=&\gamma[w(t,x)]^\beta-f(w(t,x))+(-\Delta)^sm(t,x)\\
		&\geq&\gamma[w(t,x)]^\beta- r[w(t,x)]^\beta-C_1\\
		&=& (\gamma-r)[w(t,x)]^\beta-C_1.
	\end{eqnarray*}

    Since $w(t,x')$ is decreasing with respect to $x'$, since $x_0(t)<x\leq x_\gamma(t)$, then we obtain
    \begin{eqnarray*}
    	\partial_tm(t,x)+(-\Delta)^sm(t,x)-f(m(t,x))\geq (\gamma-r)[w(t,x_\gamma(t))]^\beta-C_1=0.
    \end{eqnarray*}

\end{proof}

In the following, let's verify the supersolution inequality for $(t,x)$ with $x>x_\gamma(t)$. First, we introduce notations $q:=p(\beta-1)$ and $\sigma:=\gamma(\beta-1)t$. For some constant $K>2$ which will be determined later, let's write
\begin{eqnarray*}
	-(-\Delta)^sm(t,x)&=&{\int_{-\infty}^{\frac{x_0(t)-x}{K}  } \frac{m(t,x+z)-m(t,x)}{|z|^{1+2s}} dz     }+{\textnormal{P.V. }\int^{+\infty}_{\frac{x_0(t)-x}{K}  } \frac{m(t,x+z)-m(t,x)}{|z|^{1+2s}} dz     }\\
	&=:&I_1+I_2.
\end{eqnarray*}

Since $\beta>1$, by the definitions of $x_\gamma(t)$ and $x_0(t)$, we can find some large $\gamma_1$ which may depend on $K$ such that $\gamma_1>\gamma_0$ (where $\gamma_0$ is defined in {\bf Claim \ref{claim:2-4}}), and $x_0(t)-x<-K$ for all $\gamma\geq\gamma_1$ and all $x>x_\gamma(t)$. Since $0<m(t',x')\leq1$ for all $t'>0$ and $x'\in\mathbb{R}$, then
\begin{eqnarray*}
	I_1 \leq \int_{-\infty}^{\frac{x_0(t)-x}{K}  } \frac{1}{|z|^{1+2s}} dz=\frac{1}{2s}\cdot \left[ \frac{K}{x-x_0(t)}\right]^{2s}.
\end{eqnarray*}

By choosing $q<1$ (that is, $p(\beta-1)<1$), then $x^q\leq [x_0(t)]^q+(x-x_0(t))^q$, that is, $[x^q-[x_0(t)]^q  ]^{\frac{1}{q}}\leq x-x_0(t)$, which implies that
\begin{eqnarray*}
	\frac{1}{[x-x_0(t)]^{2s }}\leq \frac{1}{[x^q-[x_0(t)]^q]^{\frac{2s}{q}} }.
\end{eqnarray*}

Since $x>x_0(t)+K>2$, then $(x^q-1)^{\frac{1}{q}}>1$. So we get $\left[ v_0((x^q-1)^{\frac{1}{q}})  \right]^{1-\beta}=x^q-1$, which implies that 
\begin{eqnarray*}
    w(t,(x^q-1)^{\frac{1}{q}})   &=&\frac{ 1 }{\left[ \left[ v_0((x^q-1)^{\frac{1}{q}})  \right]^{1-\beta}-\gamma(\beta-1)t   \right]^{\frac{1}{\beta-1} }}\\
	&=&\frac{1}{\left[ x^q-1-\gamma(\beta-1)t   \right]^{\frac{1}{\beta-1}  }}\\
	&=&\frac{1}{\left[ x^q-[x_0(t)]^q   \right]^{\frac{1}{\beta-1}  }}.
\end{eqnarray*}

Since $\frac{2s}{q}\cdot (\beta-1)=\frac{2s}{p}$ and $x-x_0(t)>K$, then we can obtain
\begin{eqnarray*}
	I_1\leq \frac{1}{2s}\cdot K^{2s}\cdot \left[ w(t,(x^q-1)^{\frac{1}{q}})  \right]^{\frac{2s}{p}}.
\end{eqnarray*}

Since $x>1$ and $q=p(\beta-1)$, then we have
\begin{eqnarray*}
	\frac{w(t,(x^q-1)^{\frac{1}{q}}) }{w(t,x)}&=&\frac{1}{w(t,x)}\cdot \frac{1}{\left[ x^q-1-\gamma(\beta-1)t   \right]^{\frac{1}{\beta-1}  }}\\
	&=&\frac{1}{w(t,x)}\cdot\frac{1}{\left[ x^q-1- [v_0(x)]^{1-\beta}+[w(t,x)]^{1-\beta})   \right]^{\frac{1}{\beta-1}  }}\\
	&=&\frac{1}{w(t,x)}\cdot\frac{1}{\left[ x^q-1- x^{-p(1-\beta)}+[w(t,x)]^{1-\beta})   \right]^{\frac{1}{\beta-1}  }}\\
	&=&\frac{1}{w(t,x)}\cdot\frac{1}{\left[ -1+[w(t,x)]^{1-\beta})   \right]^{\frac{1}{\beta-1}  }}\\
	&=&\frac{ 1  }{\left[ 1-[w(t,x)  ]^{\beta-1}    \right]^{ \frac{1}{\beta-1 }}}.
\end{eqnarray*}

Since $x>x_\gamma(t)$, and $w(t',x')$ is decreasing with respect to $x'$, then
\begin{eqnarray*}
	w(t,x)<w(t,x_\gamma(t))=\left[ \frac{C_1}{\gamma-r}  \right]^{\frac{1}{\beta}}.
\end{eqnarray*}

So we can find some large $\gamma_2>\gamma_1$ such that for all $\gamma>\gamma_2$, we have $\left[ \frac{C_1}{\gamma-r}  \right]^{\frac{1}{\beta}}\leq 1$, which implies that
\begin{eqnarray*}
	I_1\leq C_2 K^{2s}\cdot [w(t,x)]^{\frac{2s}{p}}.
\end{eqnarray*}

For $I_2$, since $x_0(t)-x<-K$, then we have
\begin{eqnarray*}
	I_2&=&{\int_{\frac{x_0(t)-x}{K}  }^{-1} \frac{m(t,x+z)-m(t,x)}{|z|^{1+2s}} dz  }  +{\textnormal{P.V. }\int_{-1}^1 \frac{m(t,x+z)-m(t,x)}{|z|^{1+2s}} dz  }+{  \int_1^{+\infty} \frac{m(t,x+z)-m(t,x)}{|z|^{1+2s}} dz }\\
	&=:&I_3+I_4+I_5.
\end{eqnarray*}

For $I_3$, for all $\frac{x_0(t)-x}{K}\leq z\leq -1$, since $x_0(t)>1$ and $K>2$, then we have $x+z>\frac{x}{2}>1$. By changing variables, $z=xu$, then we have
\begin{eqnarray*}
	I_3&=&\int_{\frac{x_0(t)-x}{K}  }^{-1} \frac{w(t,x+z)-w(t,x)}{|z|^{1+2s}} dz\\
	&=&\int_{\frac{x_0(t)-x}{Kx}  }^{-\frac{1}{x}} \frac{w(t,x+xu)-w(t,x)}{|xu|^{1+2s}} \cdot xdu\\
	&=&xw(t,x)\int_{\frac{x_0(t)-x}{Kx}  }^{-\frac{1}{x}} \frac{1}{|xu|^{1+2s}}\cdot\left[\frac{w(t,x+xu)}{w(t,x)}-1   \right]du\\
	&=&xw(t,x)\int_{\frac{x_0(t)-x}{Kx}  }^{-\frac{1}{x}} \frac{1}{|xu|^{1+2s}}\cdot\left[ \frac{[(x+xu)^{p(\beta-1)}-(\beta-1)\gamma t]^{\frac{1}{1-\beta}}  }{[x^{p(\beta-1)}-(\beta-1)\gamma t]^{\frac{1}{1-\beta}}}   -1   \right]du\\
	&=&xw(t,x)\int_{\frac{x_0(t)-x}{Kx}  }^{-\frac{1}{x}}\frac{1}{|xu|^{1+2s}}\cdot\left[ \frac{1}{ \left[ \frac{(1+u)^q-1}{1-\frac{\sigma}{x^q}} +1  \right]^{\frac{p}{q}}  }    -1\right] du.
\end{eqnarray*}

\begin{claimd}\label{claim:2-5}
	For $q<1$, there exists some $K(q)>0$ such that for all $t>0$, all $K\geq K(q)$, all $x>x_0(t)$, and all $u\in\left[\frac{x_0(t)-x}{Kx} ,0  \right]$, we have
	\[ \frac{(1+u)^q-1}{1-\frac{\sigma}{x^q}}-1\geq -\frac{1}{2}.  \]
	
\end{claimd}

\begin{proof}
	The proof goes identically with the one in \cite{Alfaro2017}. 
\end{proof}

By {\bf Claim \ref{claim:2-5}} and Lagrange's Mean Value Theorem, then there exists some constant $C_3>0$ such that
\begin{eqnarray*}
	\frac{1}{ \left[ \frac{(1+u)^q-1}{1-\frac{\sigma}{x^q}} +1  \right]^{\frac{p}{q}}  }    -1&\leq&-\frac{p}{q}[1+C_3]\cdot \frac{(1+u)^q-1}{1-\frac{\sigma}{x^q}}\\
	&=&\frac{p}{q}[1+C_3]\cdot x^q\cdot \frac{1-(1+u)^q}{x^q-\sigma}.
\end{eqnarray*}

    Since $[w(t,x)]^{1-\beta}=x^q-\sigma$ and $0<q<1$, then
    \begin{eqnarray*}
    	\frac{1}{ \left[ \frac{(1+u)^q-1}{1-\frac{\sigma}{x^q}} +1  \right]^{\frac{p}{q}}  }    -1&\leq&\frac{p}{q}[1+C_3]\cdot x^q [w(t,x)]^{\beta-1}[1-(1+u)^q] \\
    	&\leq&\frac{p}{q}[1+C_3]\cdot x^q [w(t,x)]^{\beta-1}\cdot |u|^q\\
    	&=&C_4\cdot x^q [w(t,x)]^{\beta-1}\cdot |u|^q.
    \end{eqnarray*}

So we have
\begin{eqnarray*}
	I_3&\leq& xw(t,x)\int_{\frac{x_0(t)-x}{Kx}  }^{-\frac{1}{x}}\frac{1}{|xu|^{1+2s}}\cdot C_4\cdot x^q [w(t,x)]^{\beta-1}\cdot |u|^qdu\\
	&=&C_4[w(t,x)]^\beta \int_{\frac{x_0(t)-x}{Kx}  }^{-\frac{1}{x}}\frac{1}{|xu|^{1+2s}}\cdot |xu|^q\cdot x du\\
	&=&C_4[w(t,x)]^\beta\int_{\frac{x_0(t)-x}{K}}^{-1} \frac{1}{|z|^{1+2s-q}} dz.
\end{eqnarray*}

Let's take $q$ such that $2s-q>0$, that is, $2s>q=p(\beta-1)$, then we have
\begin{eqnarray*}
	I_3\leq C_4[w(t,x)]^\beta\int_{-\infty}^{-1} \frac{1}{|z|^{1+2s-q}} dz=C_5[w(t,x)]^\beta.
\end{eqnarray*}

For $I_4$, it's easy to see that
\[ I_4=\frac{1}{2}\int_{-1}^1 \frac{m(t,x+z)+m(t,x-z)-2m(t,x) }{|z|^{1+2s}} dz.  \]

By the proof of {\bf Claim \ref{claim:2-1}}, since $p+1\geqq p\beta$ and $x>1$, then  we know that
\begin{eqnarray*}
	|\partial^2_{xx}m(t,x)|&\leq& p^2\beta[m(t,x)]^{2\beta-1}+(p+1-p\beta)[m(t,x)]^{\beta}.
\end{eqnarray*}

For any $z\in [-1,1]$, since $x>2$, then $x+z>1$, which implies that 
\begin{eqnarray*}
	|\partial^2_{xx}m(t,x)|&\leq& p^2\beta[w(t,x)]^{2\beta-1}+(p+1-p\beta)[w(t,x)]^{\beta}.
\end{eqnarray*}

Since $\beta>1$ and $0<w(t,x)\leq 1$, then there exists some $C_6>0$ such that
\begin{eqnarray*}
	|\partial^2_{xx}m(t,x)|&\leq& C_6[w(t,x)]^{\beta}.
\end{eqnarray*}

So we know that
\begin{eqnarray*}
	I_4\leq\frac{1}{2}\int_{-1}^1 \frac{C_6[w(t,x)]^\beta |z|^2   }{|z|^{1+2s}} dz=C_7[w(t,x)]^\beta.
\end{eqnarray*}

For $I_5$, for all $z\geq0$, since $m(t',x')$ is decreasing with respect to $x'$, then
\[I_5\leq 0.  \]

So we have
	\begin{eqnarray*}
	\partial_tm(t,x)+(-\Delta)^sm(t,x)-f(m(t,x))&=&\partial_tw(t,x)+(-\Delta)^sm(t,x)-f(w(t,x))\\
	&=&\gamma[w(t,x)]^\beta-f(w(t,x))+(-\Delta)^sm(t,x)\\
	&\geq&\gamma[w(t,x)]^\beta- r[w(t,x)]^\beta-I_1-I_2\\
	&=& (\gamma-r)[w(t,x)]^\beta-I_1-I_3-I_4-I_5\\
	&\geq&[w(t,x)]^\beta[\gamma-r  C_2 K^{2s}\cdot [w(t,x)]^{\frac{2s}{p}-\beta}- C_5-C_7 ].
\end{eqnarray*}

Let's take $p=\frac{2s}{\beta}$ (in this case, we have $p+1\geq p\beta$ and $2s>p(\beta-1 )$), when $\gamma$ is large enough, we have
\[ \partial_tm(t,x)+(-\Delta)^sm(t,x)-f(m(t,x))\geq0.  \]

In summary, we can conclude that $m(t,x)$ is a supersolution to the problem \eqref{eqn:1-1}. By the comparision principle, then
\[u(t,x)\leq m(t,x),\qquad\forall t>0,\ \forall x\in\mathbb{R}.  \]

So for any $\lambda\in(0,1)$ and any $x\in\Gamma_\lambda(t)$, then 
\[ \lambda\leq u(t,x) \leq m(t,x)=\frac{1}{\left[x^{\frac{2s(\beta-1)}{\beta} }-\gamma(\beta-1)t  \right]^{\frac{1}{\beta-1}}}. \]

Since $\beta>1$, hence we get
\[ x\leq \left[ \left(\frac{1}{\lambda}  \right)^{\beta-1}+\gamma(\beta-1)t    \right]^{\frac{\beta}{2s(\beta-1)}  }.   \]

So when $T_\lambda\gg1$, we have
\begin{eqnarray*}
	x_\lambda(t)\leq C(\lambda)\cdot t^{\frac{\beta}{2s(\beta-1)}}.
\end{eqnarray*}

\section{Lower bound on the speed of the super level sets}\label{sec-l}

\begin{proof}[Proof of {\bf Theorem \ref{thm:1-2}}]
	By {\bf Assumption \ref{ass:1-2}} on the initial data $u_0(x)$, we can construct a non-increasing function $\widetilde{u}_0(x)$ such that $\widetilde{u}_0(x)\leq u_0(x)$ for all $x\in\mathbb{R}$ and
	\begin{eqnarray*}
		\widehat{u}_0(x)=\left\{
		\begin{array}{ll}
			c_0, &\textnormal{if $x\leq-R_0-1$},\\
			0, &\textnormal{if $x\geq -R_0$},
		\end{array}
		\right.
	\end{eqnarray*}
	
	\noindent for some small $0<c_0\ll1$ and some large $R_0\gg1$.
	
	Let $v(t,x)$ be the solution of the following problem:
	\begin{eqnarray*}
		\left\{
		\begin{array}{l}
			v_t+(-\Delta)^sv=0,\quad t>0, \ x\in\mathbb{R},\\
			v(0,x)=\widehat{u}_0(x),\quad x\in\mathbb{R}.
		\end{array}
		\right.
	\end{eqnarray*}
	
	Since $f(u)\geq0$ for all $u\in[0,1]$, it's easy to see that $v(t,x)$ is a subsolution to the problem \eqref{eqn:1-1}. By the comparison principle, we have
	\[ v(t,x)\leq u(t,x),\qquad\forall t>0,\ x\in\mathbb{R}. \]
	
	Let $p_s(t,x)$ be the heat kernel for $(-\Delta)^s$, then we have
	\begin{eqnarray*}
		v(t,x)&=&\int_{\mathbb{R}} \widehat{u}_0(x-y)p_s(t,y) dy,\qquad\forall t>0,\ x\in\mathbb{R}.
	\end{eqnarray*}
	
	For the heat kernel for $(-\Delta)^s$, it's well known that there exists some constant $1>C_1>0$ such that
	\begin{eqnarray*}
		\frac{C_1}{t^{\frac{1}{2s}}[1+|t^{-\frac{1}{2s}} x|^{1+2s} ] }\leq p_s(t,x)\leq \frac{C_1^{-1}}{t^{\frac{1}{2s}}[1+|t^{-\frac{1}{2s}} x|^{1+2s} ] },\qquad\forall t>0,\ x\in\mathbb{R}.
	\end{eqnarray*}
	
	So we get
	\begin{eqnarray}
		v(t,x)&\geq&\int_{\mathbb{R}} \widehat{u}_0(x-y)\cdot 	\frac{C_1}{t^{\frac{1}{2s}}[1+|t^{-\frac{1}{2s}} y|^{1+2s} ] } dy \nonumber \\
		&\geq&\int_{x+R_0+1 }^{+\infty} \frac{ c_0\cdot C_1  }{t^{\frac{1}{2s}}[1+|t^{-\frac{1}{2s}} y|^{1+2s} ] }  dy\nonumber \\
		&=&\int_{t^{-\frac{1}{2s}}(x+R_0+1) }^{+\infty} \frac{ c_0\cdot C_1  }{1+|z|^{1+2s} }  dz \label{esti-diff}
	\end{eqnarray}
	
	In particular, we have
	$$ u(1,x)\ge \int_{x+R_0+1}^{+\infty} \frac{ c_0\cdot C_1  }{1+|z|^{1+2s}  }  dz.$$

	Thus for $x>1$, we may find a constant $C_2>0$ such that 
	$$ u(1,x)\ge \frac{C_2}{x^{2s}}.$$
	
	As a result, we can find a small enough $d>0$ such that
	\begin{equation}\label{ac:def-v0-1}
	u(1,x)\geq v(1,x) \ge \tilde{u}_0(x):= 
	\begin{cases}
	d &\text{ for }  x \le 1\\
	\frac{d}{x^{2s}} &\text{ for } x \ge 1.
	\end{cases}
	\end{equation}
	
	Hence, from the comparison principle and up to a shift in time, we only need to get the lower estimate for the case that $u(t,x)$ is the solution to the problem \eqref{eqn:1-1} with the initial data $ \tilde{u}_0$. Since $\tilde{u}_0(x)$ is decreasing, it's easy to see that $u(t,x)$ is decreasing with respect to $x$. Let $\displaystyle \lambda_0:=\int_{\frac{1}{2} }^{+\infty} \frac{ c_0\cdot C_1  }{1+|z|^{1+2s} }  dz$, $\displaystyle x_B(t):=\frac{t^{\frac{1}{2s} }}{4}$, and $x_{\lambda_0}(t)$ be such that $u(t,x_{\lambda_0}(t))=\lambda_0$. From \eqref{esti-diff}, then there exist $T_{\lambda_0}\gg1$ such that for all $t\ge T_{\lambda_0}$ we have, 
$$v(t,x_B(t))\ge \int_{\frac{1}{4}+t^{-\frac{1}{2s}}(R_0+1) }^{+\infty} \frac{ c_0\cdot C_1  }{1+|z|^{1+2s}  }  dz\ge \int_{\frac{1}{2} }^{+\infty} \frac{ c_0\cdot C_1  }{1+|z|^{1+2s}  }  dz =\lambda_0. $$

Since $\widehat{u}_0(x)$ is decreasing, it's easy to see that $v(t,x)$ is decreasing with respect to $x$. Since $u(t,x)\geq v(t,x)$ for all $t\geq0$ and all $x\geq0$, and $u(t,x)$ and $v(t,x)$ are decreasing with respect to $x$, then we can get $\frac{t^{\frac{1}{2s} }}{4}=x_B(t)\le x_{\lambda_0}(t)$. The above argument holds as well for any $0<\lambda \leq \lambda_0$
which provides the lower estimate.

It  remains to obtain a similar bound for a given $\lambda_0<\lambda <1$. To obtain the bound we can argue as in \cite{Alfaro2017}. 
So first let us prove an invasion lemma on the solution of the Cauchy problem \eqref{eqn:1-1}. Namely, 
\begin{prop}\label{prop:level-set}
For any $0<s<1$, assume that the nonlinearity $f$ satisfies {\bf Assumption \ref{ass:1-1}}, and the initial data $u_0(x)$ satisfies {\bf Assumption \ref{ass:1-2}}. Assume $\beta>1$ and $\frac{\beta}{2s(\beta-1)}>1$, and let $u(t,x)$ be the solution to the problem \eqref{eqn:1-1} with the initial data $u_0(x)$.
Then, for any $A \in \R$, 
\begin{equation}\label{invasion}
\lim_{t\to \infty}u(t,x)=1 \quad \text{ uniformly in  } (-\infty,A],
\end{equation}
and, for any $\lambda \in (0,1)$,  
\begin{equation}\label{acc-prop}
\lim_{t\to\infty}\frac{x_{\lambda}(t)}{t}=+\infty.
\end{equation} 

\end{prop}

Let us postpone for a moment the proof of {\bf Proposition \ref{prop:level-set}} and finish the proof of {\bf Theorem \ref{thm:1-2}}. Let us denote by $w(t,x)$ the solution of \eqref{eqn:1-1} starting from a nonincreasing $w_0$ such that
 \begin{equation}\label{initial-data-w}
w_0(x)= \begin{cases} \lambda_0 &\text{
if } x\leq -1
\\
0 &\text{ if  }x\geq 0.
\end{cases}
\end{equation}

It follows from  {\bf Proposition \ref{prop:level-set}}  that there is a time $\tau_{\lambda}>0$  such that
\begin{equation}
\label{w-grand}
 w(\tau_{\lambda},x)>\lambda, \quad \forall x\leq 0.
 \end{equation}
 
 On the other hand, since $\widehat{u}_0(x)$ is decreasing, it's easy to see that $v(t,x)$ is also decreasing with respect to $x$. Since $u(t,x_B(t))\geq\lambda_0$, then
 \begin{eqnarray*}
 	u(t,x)\geq v(t,x)\geq \lambda_0,\qquad\forall x\leq x_B(t).
 \end{eqnarray*}
 
  So it follows from  \eqref{initial-data-w} that
 $$
 u(T,x)\ge w_0(x-x_{\lambda_0}(T)),\quad \forall T\geq 0, \forall x\in \R.
 $$
 
So the comparison principle yields
 $$
 u(T+\tau,x)\geq w(\tau,x-x_{\lambda_0}(T)),\quad \forall T\geq 0, \forall \tau \geq 0, \forall x\in \R.
 $$
 
  In view of \eqref{w-grand}, this implies that
 $$
 u(T+\tau_{\lambda},x)>\lambda,\quad \forall T\geq 0, \forall x\leq x_{\lambda_0} (T).
 $$
 
 Hence, for $t\geq \tau _\lambda$,   the above implies  
 $$
 x_\lambda(t)\geq x_{\lambda_0}(t-\tau _\lambda)=\frac{(t-\tau _\lambda)^{\frac{1}{2s} }}{4} \geq \underline C t^{\frac{1}{2s}},
 $$
provided $t\geq T_\lambda '$, with $T_\lambda '>\tau _\lambda$ large enough. This concludes the proof of the lower estimate.

	In summary, we can conclude that for any $\lambda\in(0,1)$,  then there exist some constants $T_\lambda'>0$ and $C'(\lambda)>0$ such 
	\begin{eqnarray*}
		x_\lambda(t)\geq C'(\lambda) t^{\frac{1}{2s}  },\qquad\forall t>T_\lambda'.
	\end{eqnarray*}
	
\end{proof}

In order to prove {\bf Proposition \ref{prop:level-set}}, let us first establish the following result.
\begin{prop}[Speeds of a sequence of  bistable traveling waves]\label{prop:cinfinie}
For any $0<s<1$, assume that the nonlinearity $f$ satisfies {\bf Assumption \ref{ass:1-1}}, and the initial data $u_0(x)$ satisfies {\bf Assumption \ref{ass:1-2}}. Assume $\beta>1$ and $\frac{\beta}{2s(\beta-1)}>1$. Let $(g_{n})=(g_{\theta _n })$ be a sequence of bistable nonlinearities such that  $g_n\le  g_{n+1}\le f$ and $g_n\to f$. Let $(c_n,U_n)$ be the associated sequence of traveling waves. Then
$$
\lim_{n\to \infty }c_n=+\infty.
$$
\end{prop}

\begin{proof} Since $g_{n+1}\geq g_n$ it follows from standard sliding techniques  \cite{Berestycki1991a, Chen1997,Coville2005,Coville2006,Coville2007} that $c_{n+1}\geq c_n$.  Assume now  by contradiction that $c_n \nearrow \bar c$ for some $\bar c\in \R$.  Observe that  since $g_n\to f$ and $\int_0^1f(s)\,ds>0$, we have  $c_n\ge c_0>0$ for $n$ large enough, says, $n\ge n_0$. As a consequence, for all $n\ge n_0$, $U_n$ is smooth and since any translation of $U_n$ is a still a solution, without loss of generality, we can assume  the normalization $U_n(0)=1/2$. Now,  thanks to Helly's Theorem \cite{Brunk1956} and up to extraction, $U_n$ converges to a monotone function $\bar U$ such that $\bar U(0)=\frac{1}{2}$. Also, since $c_n<\bar c$, from the equation we get an uniform bound on $U^{'}_n, U^{''}_n$ and up to extraction, $U_n$ also  converges in  $C^2_{loc}(\R)$, and the limit has to be $\bar U'$. As a result, $\bar U$ is monotone and solves 
$$\begin{cases}
(-\Delta^{s})\bar U+\bar c\bar U'+f(\bar U)=0 \quad \text{ on } \R,\\
\bar U(-\infty)=1,\quad \bar U(0)=\frac{1}{2}, \quad \bar U(\infty)=0.
\end{cases}
$$ 

In other words, we have constructed a monostable traveling wave under assumption that  $\beta>1$ and $\frac{\beta}{2s(\beta-1)}>1$ , which is  a contradiction with the result in \cite{Gui2015}. 
\end{proof}

Equipped with this technical result we can establish {\bf Proposition \ref{prop:level-set}}.
\begin{proof} 
First, we prove \eqref{invasion} for the particular case where the initial datum $u_0$ is a smooth decreasing function such that
\begin{equation}\label{d0}
u_0(x)=\begin{cases}
d_0 &\text{ for }  x \le -1\\
0 &\text{ for }  x \ge 0,
\end{cases}
\end{equation}
for an arbitrary $0<d_0<1$. Since $u_0$ is  nonincreasing, we deduce from the comparison principle that, for all $t>0$, the function $u(t,x)$ is still decreasing in $x$.

Let us now extend $f$  by $0$ outside  the interval $[0,1]$. From \cite{Achleitner2015} and {\bf Proposition \ref{prop:cinfinie}}, there exists $0<\theta<d_0$ and  a Lipschitz bistable function $g\le f$ --- i.e. $g(0)=g(\theta)=g(1)=0$,  $g(s)<0$ in $(0,\theta)$, $g(s)>0$ in $(\theta,1)$, and $g'(0)<0$, $g'(1)<0$, $g'(\theta)>0$---  such that there exists a smooth decreasing  function $U_\theta$ and  $c_\theta>0$ verifying
\begin{align*}
&(-\Delta)^{s}U_\theta +c_{\theta} U_\theta'+g(U_\theta)=0 \quad \text{ on } \R,    \\
&U_\theta(-\infty)=1, \qquad U_\theta(\infty)=0.
\end{align*}

Let us now consider $v(t,x)$ the solution of the Cauchy problem
\begin{align*}
&\partial_t v(t,x) =-(-\Delta)^s v(t,x) + g(v(t,x)) \quad \text{ for  } t>0, x\in \R,    \\
&v(0,x)=u_0(x).
\end{align*}
Since $g\le f$, $v$ is a subsolution of the Cauchy problem \eqref{eqn:1-1} and by the comparison principle, $v(t,x)\le u(t,x)$ for all $t>0$ and $x \in \R$.

Now, thanks to the global asymptotic stability result \cite[Theorem 3.1]{Achleitner2015}, since $d_0>\theta$, then we know that there exists $\xi\in\R, C_0>0$  and $\kappa>0$ such that for all $t\ge 0$
$$
\|v(t,\cdot)-U_{\theta}(\cdot-c_\theta t+\xi)\|_{L^\infty} \le C_0 e^{-\kappa t}.
$$

Therefore,  for all $t>0$ and $x \in \R$, we have
$$
u(t,x)\ge v(t,x)\ge U_{\theta}(x-c_\theta t+\xi)-C_0e^{-\kappa t}.
$$

Since $c_\theta>0$,  by sending $t \to \infty$, we get $1\ge \liminf_{t\to \infty}u(t,x)\ge \lim_{t\to \infty}[U_{\theta}(x-c_\theta t+\xi)-C_0e^{-\kappa t}]= 1$. As a result, for all $x\in\mathbb{R}$, we have $1\ge \limsup_{t\to \infty} u(t,x)\ge \liminf_{t\to \infty} u(t,x)= 1$, which implies that $u(t,x)\to 1$  as $t \to \infty$.  Since $u(t,x)$ is decreasing in $x$, then the convergence is uniform on any set $(-\infty,A]$. This concludes the proof of \eqref{invasion} for our particular initial datum.

For a generic initial data satisfying {\bf Assumption \ref{ass:1-2}}, we can always, up to a shift in space, construct a smooth decreasing $\tilde{u_0}$ satisfying \eqref{d0} and $\tilde{u_0}\leq u_0$. Since the solution $\tilde u(t,x)$ of the Cauchy problem starting from $\tilde{u_0}$ satisfies \eqref{invasion}, so does $u(t,x)$ thanks to the comparison principle. 
\end{proof}

\section{Another better lower bound on the speed of the super level sets}\label{sec3}

Here we prove another lower bound on the speed of $x_\lambda(t)$ when $1<\beta<2$ and$\frac{1}{2s(\beta-1)}>1$ (notice that $\frac{1}{2s(\beta-1)}>\frac{1}{2s}$ if and only if $1<\beta<2$). In the whole of this section, let's assume the conditions in {\bf Theorem \ref{thm:1-3}} hold. As above to measure the  acceleration, we use a subsolution that fills the space with a superlinear speed. The construction of this subsolution is an adaptation of  the one proposed by Alfaro and Coville  \cite{Alfaro2017} for a  nonlocal diffusion with an integrable kernel. It essentially contains three steps.

\medskip

\noindent {\bf Step one.} It consists in using the diffusion to gain an algebraic tail at time $t=1$.

 From the proof of {\bf Theorem \ref{thm:1-2}}, we have known that we can find a small enough $d>0$ such that
\begin{equation}\label{ac:def-v0}
u(1,x)\geq v(1,x) \ge v_0(x):= 
\begin{cases}
d &\text{ for }  x \le 1\\
\frac{d}{x^{2s}} &\text{ for } x \ge 1.
\end{cases}
\end{equation}
Hence, from the comparison principle and up to a shift in time, it is enough to 
prove the lower estimate for $u(t,x)$ which is the solution starting from the initial data $v_0$, which we do below.

\medskip

\noindent {\bf Step two.} Here we construct explicitly the subsolution that we are considering.

Following Alfaro-Coville \cite{Alfaro2017} let us consider the  function $g(y):=y(1-By)$, with $B>\frac{1}{2d}$, it's easy to see that  $g(y)\le 0$ if and only if $0\leq y\leq \frac{1}{B} $, and  $g(y)\le g(\frac{1}{2B})=\frac{1}{4B}<d$ for all $y\in\mathbb{R}$.

As in the previous subsection,  for any $\gamma>0$, let  $w(\cdot,x)$ denote the solution to the Cauchy problem
\begin{eqnarray*}
	\left\{
	\begin{array}{l}
		\displaystyle \frac{dw}{dt}(t,x)=\gamma [w(t,x)]^{\beta},\\
		w(0,x)=v_0(x).
	\end{array}
	\right.
\end{eqnarray*}

That is
 $$
 w(t,x)=\frac{1}{\left[ [v_0(x)]^{1-\beta} -\gamma(\beta -1) t\right]^{\frac{1}{\beta -1}}}, 
 $$
where $v_0$ is defined in  \eqref{ac:def-v0}. 

Notice that $w(t,x)$ is not defined for all times. When $x\le 1$, $w(t,x)$ is defined for $t\in[0,\frac{1}{d^{\beta -1}\gamma(\beta-1)})$, whereas for $x>1$,   $w(t,x)$ is defined for $t\in\left[0,T(x):=\frac {x^{2s(\beta -1)}}{d^{\beta -1}\gamma(\beta-1)}\right)$.  Let us define
 \begin{equation}\label{ac-def-xB}
 x_{B}(t):=d^{\frac{1}{2s}} \left[\left(2B \right)^{\beta-1}+\gamma (\beta -1)t \right]^{\frac{1}{2s(\beta-1)}}.
 \end{equation}
 
Since $B>\frac{1}{2d}$ and $\beta>1$, then $x_B(t)>1$ and $w(t,x_B(t))=\frac{1}{2B}$. For $x<1$ and $0<t<\frac{1}{d^{\beta-1}\gamma(\beta-1)}$, since $v_0(x')=d$ for all $x'\leq 1$, then we have $\partial_xw(t,x)=\partial_{xx} w(t,x)=0$. For  $x>1$ and $0< t< T(x)$, we compute
	\begin{eqnarray*}
		\partial_xw(t,x)&=&\frac{1}{1-\beta}\cdot [ [v_0(x)]^{1-\beta}-\gamma(\beta-1)t  ]^{\frac{1}{1-\beta} -1}\cdot (1-\beta)[v_0(x)]^{-\beta}\cdot v_0'(x)\\
		&=&-2d^{1-\beta}s[w(t,x)]^\beta\cdot x^{2s\beta-2s-1}\\
		&<&0\\
		\partial_{xx} w(t,x)&=&-2d^{1-\beta}s\cdot[ \beta [w(t,x)]^{\beta-1}\cdot \partial_x w(t,x)  \cdot x^{2s\beta-2s-1}+[w(t,x)]^\beta\cdot (2s\beta-2s-1)x^{2s\beta-2s-2}  ]\\
		&=&-2d^{1-\beta}s\cdot\left[ \beta [w(t,x)]^{\beta-1}\cdot \left(-2d^{1-\beta}s[w(t,x)]^{\beta}\cdot x^{2s\beta-2s-1}\right)  \cdot x^{2s\beta-2s-1}\right.\\
		&&\left.+[w(t,x)]^{\beta}\cdot (2s\beta-2s-1)x^{2s\beta-2s-2}  \right]\\
		&=&2d^{1-\beta}s [w(t,x)]^{\beta}\cdot x^{2s\beta-2s-2}\cdot[ 2d^{1-\beta}s\beta\cdot [w(t,x)]^{\beta-1}x^{2s\beta-2s}+2s+1-2s\beta   ]\\
		&=&2d^{1-\beta}s [w(t,x)]^{\beta}\cdot x^{2s\beta-2s-2}\cdot[ 2d^{1-\beta}s\beta\cdot [[v_0(x)]^{1-\beta}-\gamma(\beta-1)t]^{-1} x^{2s\beta-2s}+2s+1-2s\beta   ]\\
		&>&2d^{1-\beta}s [w(t,x)]^{\beta}\cdot x^{2s\beta-2s-2}\cdot[ 2d^{1-\beta}s\beta\cdot [v_0(x)]^{\beta-1} x^{2s\beta-2s}+2s+1-2s\beta   ]\\
		&=&2d^{1-\beta}s [w(t,x)]^{\beta}\cdot x^{2s\beta-2s-2}\cdot\left[ 2d^{1-\beta}s\beta\cdot \left(\frac{d}{x^{2s}} \right)^{\beta-1} x^{2s\beta-2s}+2s+1-2s\beta   \right]\\
		&=&2d^{1-\beta}s [w(t,x)]^{\beta}\cdot x^{2s\beta-2s-2}\cdot(2s+1)\\
		&>&0.
\end{eqnarray*}

In the first inequality of the computation of $\partial_{xx}w(t,x)$, we used the condition $\beta>1$ and $\gamma>0$. Hence, for any $t>0$, the function $w(t,\cdot)$ is decreasing and convex with respect to the variable $x$.

Let us now define the continuous function
 $$
 m(t,x):= 
 \begin{cases}
\frac{1}{4B} &\text{ for }  x \le x_{B}(t)\\
g(w(t,x)) &\text{ for } x_B(t)<x.
\end{cases}
$$

Note that by the construction of $m(t,x)$ for all $t>0$, it's easy to see that the function $m(t,x)$ is $C^{1,1}(\R)$ in $x$, and $$\partial_xm(t,x)=\partial_xw(t,x)(1-2Bw(t,x))^+,$$ which is a Lipschitz function.

Observe that: when $x>x_B(0)=(2d B)^{\frac{1}{2s}}>1$, we have $m(0,x)=g(w(0,x))=g(v_0(x))\leq v_0(x)$; when $x<1$, we have $m(0,x)=\frac{1}{4B}< d=v_0(x)$; when $1\leq x\leq x_B(0)=(2dB)^{\frac{1}{2s}}$, we have $v_0(x)=\frac{d}{x^{2s}} \geq\frac{d}{2dB}=\frac{1}{2B}>\frac{1}{4B}=m(0,x)$. Hence $m(0,x)\leq v_0(x)$ for all $x\in \R$. Let us now show that $m(t,x)$ is a subsolution  for some appropriate choices of $\gamma$ and $B$. 
 
By the definition of $m(t,x)$, we have  $\partial_tm(t,x)=\gamma w^{\beta}(t,x)(1-2Bw(t,x))^+$, therefore we get  
\begin{equation}\label{eq-dtm}
\partial _t m(t,x)\le 
 \begin{cases}
0  &\text{ for }  x \le x_B(t)-1\\
\gamma w^\beta(t,x)&\text{ for }  x_B(t)-1<x\\
\end{cases}
\end{equation}

Since $f$ satisfies $f(u)\geq r_1u^{\beta}$ as $u\rightarrow0^+$, then there exists a small $r_2>0$ such that $f(u)\geq r_2 u^\beta (1-u)$ for all $0\leq u \leq 1$. When $x\leq x_B(t)$, then $m(t,x)=\frac{1}{4B}$. Since $w(t,x_B(t))=\frac{1}{2B}$, then $f(m(t,x))\geq r_2[m(t,x)]^{\beta}[1-m(t,x)]=r_2\left[ \frac{1}{2}w(t,x_B(t))  \right]^{\beta}\left(1-\frac{1}{4B}\right)=\frac{r_2}{2^{\beta}}\left( 1-\frac{1}{4B} \right) [w(t,x_B(t))]^{\beta}$. When $x>x_B(t)$, since $0\leq g(y)\leq \frac{1}{4B}$ and $w(t,x)\leq \frac{1}{2B}$, then $f(m(t,x))\geq r_2[w(t,x)(1-Bw(t,x))  ]^{\beta}[1-g(w(t,x))]\geq r[w(t,x)(1-B\cdot \frac{1}{2B})  ]^{\beta}[1-\frac{1}{4B}]=\frac{r}{2^{\beta}}\left( 1-\frac{1}{4B} \right) [w(t,x)]^{\beta}$. In summary, we have
\begin{equation}\label{ac-eq-f}
f(m(t,x))\geq 
 \begin{cases}
C_0[w(t,x_B(t))]^\beta  &\text{ for }  x \le x_B(t)\\
C_0[w(t,x)]^{\beta} &\text{ for } x > x_B(t),
\end{cases}
\end{equation}
where  $C_0:=\frac{r}{2^\beta}\left(1- \frac{1}{4B}\right)$.

Let us now derive some estimate on the fractional diffusion term $ (-\Delta)^{s}m(t,x)$ on the three regions $x\leq x_B(t)-1$, $x_B(t)-1<x<x_B(t)+1$ and $x>x_B(t)+1$. For simplicity of the presentation, we dedicated a subsection to each region and let us start with the region $x\le x_B(t)-1$.

\subsubsection*{\underline{$\bullet$ When $x\le x_B(t)-1$}:}~\\

In this region of space,  we  claim that:
\begin{claimd} \label{subsolcla1}
	\begin{enumerate}
		\item[(a)] If $\frac{1}{2}< s<1$, then there exists $C_3>0$ such that for all $x\le x_B(t)-1$, we have
		$$ (-\Delta)^{s}m(t,x)\le -C_3 v_0'(x_B(t)) [v_0(x_B(t))]^{-\beta} [w(t,x_B(t)) ]^{\beta}.$$
		
		\item[(b)] If $0<s\le \frac{1}{2}$, for large enough $B\gg1$, then there exists $C_3>0$ such that for all $x\le x_B(t)-1$, we have
		$$ (-\Delta)^{s}m(t,x)\le \frac{C_3}{B^2}. $$ 
		
		    
	\end{enumerate}
	  
\end{claimd}

Note that the singularity here play a major role and the estimate strongly depends on the value of $s$. 
\begin{proof}
For $x\le x_B(t)-1$, since $m(t,y)=\frac{1}{4B}=m(t,x_B(t))$ for all $y\leq x_B(t)$, then we have 
\begin{eqnarray*}
(-\Delta)^{s}m(t,x)= \int_{x_B(t)}^{+\infty}\frac{m(t,x_{B}(t))-m(t,y)}{|x-y|^{1+2s}}dy.
\end{eqnarray*}

We now treat separately the following two situations: $\frac{1}{2}<s<1$, $0<s\le \frac{1}{2}$. 

{\bf Case I}: $\frac{1}{2}<s<1$. By using the Fundamental Theorem of Calculus, then we have
\begin{align*}
-(-\Delta)^{s}m(t,x)&=   \int_{x_B(t)}^{+\infty}\int_0^1\frac{(y-x_B(t))}{|x-y|^{1+2s}}\partial_xm(t,x_B(t)+\tau(y-x_B(t)))\,d\tau dy,\\
&=\int_{0}^{\infty}\int_0^1\frac{z}{|x-x_B(t)-z|^{1+2s}} \partial_xm(t,x_B(t)+\tau z) \,d\tau dz.
\end{align*} 

 Now, since $w(t,\cdot)$ is a positive, decreasing and convex function, for any $z,\tau>0$, then we have
\begin{eqnarray*}
	\partial_xm(t,x_B(t)+\tau z)&=&\partial_xw(t,x_B(t)+\tau z)\left(1 -2Bw(t,x_B(t)+\tau z)\right)\\
	&\geq&\partial_xw(t,x_B(t)).
\end{eqnarray*}

So we can obtain that
$$
(-\Delta)^{s}m(t,x)\ge \partial_xw(t,x_B(t)) \int_{0}^{\infty}\frac{z}{|x-x_B(t)-z|^{1+2s}}dz.
$$

For any $z>0$, since $x<x_B(t)-1 $, then $x-x_B(t)-z<-1-z<0$, which implies that $|x-x_B(t)-z|\geq |1+z|>0$. Since $\frac{1}{2}<s<1$, then we have  
$$\int_{0}^{\infty}\frac{z}{|x-x_B(t)-z|^{1+2s}}dz<C_3:=\int_{0}^{\infty}\frac{z}{|1+z|^{1+2s}}dz<+\infty.$$

As a result 
 \begin{equation*}
-(-\Delta)^{s}m(t,x)\ge C_3 \partial_x w(t,x_B(t))=-C_3 v_0'(x_B(t)) [v_0(x_B(t))]^{-\beta} [w(t,x_B(t)) ]^{\beta},\quad \forall x\leq x_B(t)-1, 
 \end{equation*}
 
 {\bf Case II}: $0<s\le \frac{1}{2} $. In this  situation, the previous argumentation fails and we argue as follows.
  Since $w(t,y)\geq0$ for all $t$ and all $y$, for any constant $R>1$ which will be determined later, we have
 \begin{eqnarray*}
 	(-\Delta)^{s}m(t,x)&=& \int_{x_B(t)}^{+\infty}\frac{m(t,x_{B}(t))-m(t,y)}{|x-y|^{1+2s}}dy \\
 	&=&\int_{x_B(t)}^{+\infty}\frac{w(t,x_B(t))-B[w(t,x_B(t))]^2-w(t,y)+B[w(t,y)]^2  }{|x-y|^{1+2s}}dy \\
 	&=&\int_{x_B(t)}^{+\infty} \frac{[w(t,x_B(t)) -w(t,y)  ][ 1-B[w(t,x_B(t))+w(t,y)   ] ]  }{|x-y|^{1+2s}} dy\\
 	&\leq&\int_{x_B(t)}^{+\infty} \frac{w(t,x_B(t)) -w(t,y)   }{|x-y|^{1+2s}} dy\qquad\textnormal{Since $w(t,y)\geq0$ }\\
 	&\leq&\int_{x_B(t)}^{+\infty} \frac{w(t,x_B(t)) -w(t,y)   }{|y-x_B(t)+1|^{1+2s}} dy\\
 	&=&\int_{0}^{+\infty} \frac{ w(t,x_B(t))-w(t, x_B(t)+z) }{|1+z|^{1+2s}} dz\\
 	&=&\int_{0}^{R} \frac{ w(t,x_B(t))-w(t, x_B(t)+z) }{|1+z|^{1+2s}} dz+\int_{R}^{+\infty} \frac{ w(t,x_B(t))-w(t, x_B(t)+z) }{|1+z|^{1+2s}} dz\\
 	&=&I_1+I_2.
 \end{eqnarray*}
 
Let us now estimate $I_1$ and $I_2$. Since $w(t,y)\geq0$ for all $t$ and all $y$, for $I_2$ we have 
 \begin{eqnarray}
 	I_2&\leq&\int_{R}^{+\infty} \frac{ w(t,x_B(t)) }{|1+z|^{1+2s}} dz \nonumber \\
 	&\leq &\frac{1}{2B}\cdot \int_R^{+\infty} \frac{1}{z^{1+s}  } dz \nonumber \\
 	&=&\frac{1}{2B}\cdot \frac{1}{s}\cdot \frac{1}{R^{s}}.\label{com1-1}
 \end{eqnarray}
 
 On the other hand, by using the Fundamental Theorem of Calculus and the convexity of $w(t,y)$ is convex with respect to $y$, we get for $I_1$
 \begin{align*}
 	I_1=\int_{0}^{R}\int_0^1 \frac{ -\partial_x w(t,x_B(t)+\tau z )\cdot z }{|1+z|^{1+2s}} d\tau dz, &\leq\int_{0}^{R}\int_0^1 \frac{ -\partial_x w(t,x_B(t) )\cdot z }{|1+z|^{1+2s}} d\tau dz  \\
 	&\le -\partial_x w(t,x_B(t)) \int_0^R \frac{z}{ |1+z|^{1+2s} } dz.  \\
 \end{align*}
Thus, by using the definition of $\partial_xw(t,x_B(t))$, $R>1$ and since $|y|^{2s}>|y|^s$ in $(1,R)$ we get 
%
%
\begin{eqnarray*}
	I_1&\leq& -\partial_x w(t,x_B(t)) \int_0^R \frac{z}{ |1+z|^{1+2s} } dz.  \\
	&\leq& -\partial_x w(t,x_B(t))\int_1^{2R} \frac{1}{y^{2s}} dy\\
	&\leq& -\partial_x w(t,x_B(t))\int_1^{2R} \frac{1}{y^{s}} dy\\
	I_1&\le &2d^{1-\beta}s\left(\frac{1}{2B}\right)^\beta\cdot [x_B(t)]^{2s\beta-2s-1} \cdot \frac{1}{1-s}\cdot (2R)^{1-s},
\end{eqnarray*}

%
%

 which  using that $x_B(t)\geq d^{\frac{1}{2s}} (2B)^{\frac{1}{2s}}$ and  $2s\beta-2s-1<0$  enforces that
 \begin{eqnarray*}
 	I_1&\leq& 2d^{1-\beta}s\left(\frac{1}{2B}\right)^\beta\cdot [d^{\frac{1}{2s}} (2B)^{\frac{1}{2s}}]^{2s\beta-2s-1} \cdot \frac{1}{1-s}\cdot (2R)^{1-s}\\
 	&=&C_{3,1} B^{-\left( 1+\frac{1}{2s} \right)} R^{1-s}.
 \end{eqnarray*}
 
 Combining the latter estimate with \eqref{com1-1}, then we have
 \begin{eqnarray*}
 	(-\Delta)^{s}m(t,x)\leq I_1+I_2\leq C_{3,1}B^{-\left( 1+\frac{1}{2s} \right)} R^{1-s}+\frac{1}{2B}\cdot \frac{1}{s}\cdot \frac{1}{R^{s}}.
 \end{eqnarray*}
 
 By taking $R$ such that $C_{3,1} B^{-\left( 1+\frac{1}{2s} \right)} R^{1-s}=\frac{1}{2B}\cdot \frac{1}{s}\cdot \frac{1}{R^{s}}$, that is, $R=\frac{1}{2s C_{3,1}} B^{\frac{1}{2s}}$,  we then achieve 
 \begin{eqnarray*}
 (-\Delta)^{s}m(t,x)\leq \frac{C_3}{B^2}.
 \end{eqnarray*}

%
%

\end{proof}

Let us now obtain some estimate in the region $x\ge x_B(t)+1$.

\subsubsection*{\underline{$\bullet$ When $x\ge x_B(t)+1$}:}~\\

In this region,  we  claim that 

\begin{claimd}\label{subsolcla2}
	\begin{enumerate}
	    \item[(a)] If $\frac{1}{2}<s<1$,  then there exists positive constant $C_4$ such that for all $x\ge x_B(t)+1$, we have
	    $$ (-\Delta)^{s}m(t,x)\le -C_4 v_0'(x) [v_0(x)]^{-\beta}[w(t,x)]^{\beta}.$$ 
	    
	    \item[(b)] If $0<s\le\frac{1}{2}$, for large enough $B\gg1$, then there exists positive constant $C_4$ such that for all $x\ge x_B(t)+1$ , we have
	    $$ (-\Delta)^{s}m(t,x)\le -C_4\partial_xw(t,x)+C_4[w(t,x)]^{1-2s+2s\beta} x^{2s(2s\beta-2s-1)}.$$ 
	    
%
	\end{enumerate}
	
\end{claimd}

\begin{proof}
	First, we have
	\begin{eqnarray*}
		(-\Delta)^{s}m(t,x)&=&\textnormal{P.V. }\int_{\mathbb{R}} \frac{m(t,x)-m(t,y)}{|x-y|^{1+2s}} dy\\
		&=&\int_{-\infty}^{x_B(t)}\frac{m(t,x)-m(t,y)}{|x-y|^{1+2s}} dy+\textnormal{P.V. }\int_{x_B(t)}^{+\infty} \frac{m(t,x)-m(t,y)}{|x-y|^{1+2s}} dy\\
		&=:&I_1+I_2
	\end{eqnarray*}
	
	For  $I_1$, since $\partial_xm(t,x)=\partial_xw(t,x)(1-2Bw(t,x))^+$, then $m(t,\cdot)$ is decreasing. Since $x\geq x_B(t)+1$, then
	\begin{eqnarray}
		I_1\leq0.\label{com2-1}
	\end{eqnarray}
	
	For $I_2$, we have
	\begin{align*}
	I_2& =  \textnormal{P.V. } \int_{x_B(t)-x}^{\infty}\frac{m(t,x)-m(t,x+z) }{|z|^{1+2s}}dz \\    
	& =   \int_{x_B(t)-x}^{-\frac{1}{2}}\frac{m(t,x)-m(t,x+z)}{|z|^{1+2s}}dz +\textnormal{P.V. }  \int_{-\frac{1}{2}}^{\frac{1}{2}}\frac{m(t,x)-m(t,x+z) }{|z|^{1+2s}}dz+\int_{\frac{1}{2}}^{+\infty}\frac{m(t,x)-m(t,x+z) }{|z|^{1+2s}}dz\\
	&=: I_3+I_4+I_5.  
	\end{align*}
	
	Since $m(t,\cdot)$ is decreasing and $x\ge x_B(t)+1$,  it's easy to see that
	\begin{eqnarray}
	I_3\leq0.\label{com2-2}
	\end{eqnarray}
	
	For $I_4$, since the function $m(t,\cdot)$ is smooth (at least $C^2$) on  $(x-\frac{1}{2},x+\frac{1}{2})$, then we have 
	$$I_4=\frac{1}{2}  \int_{-\frac{1}{2}}^{\frac{1}{2}}\frac{m(t,x+z)+m(t,x-z)-2m(t,x)}{|z|^{1+2s}}dz \leq C_{4,1}|\partial_{xx}m(t,x)|$$
	
	A direct computation of $\partial_{xx}m(t,x)$ shows that $|\partial_{xx}m(t,x)|\le -C_{4,2}|\partial_xw(t,x)|$, which implies that 
	\begin{eqnarray}
		I_4\leq -C_{4,3}\partial_{x}w(t,x). \label{com2-3}
	\end{eqnarray}
	To complete our proof we need to estimate $I_5$ and  to do so we treat separately the following two situations: $\frac{1}{2}<s<1$, $0<s\le \frac{1}{2}$.
	
	{\bf Case I}: $\frac{1}{2}<s<1$. In this situation, by using the Fundamental Theorem of Calculus, we obtain 
$$I_5=-\int_{\frac{1}{2}}^{+\infty}\int_0^1\frac{z\partial_x w(t,x+\tau z)\left(1 -2Bw(t,x+\tau z)\right)}{|z|^{1+2s}} \,dzd\tau ,$$
which by the convexity and the monotonicity of $w(t,\cdot)$ and since $\frac{1}{2}<s<1$, then yields 
\begin{align*}
I_5&\leq -\int_{\frac{1}{2}}^{\infty}\int_0^1\frac{z}{|z|^{1+2s}}  \partial_x w(t,x+\tau z)\left(1 -2Bw(t,x+\tau z)\right) \,dzd\tau,  \\
&\le  -\partial_x w(t,x)\int_{\frac{1}{2}}^\infty\frac{z}{|z|^{1+2s}}\,dz,\\
&=  -\partial_x w(t,x)\frac{1}{2^{1-2s}(2s-1)}.
\end{align*}

By combining the latter with \eqref{com2-1}, \eqref{com2-2}, and \eqref{com2-3}, we can therefore  find  a  constant $C_4>0$ such that 
\[ (-\Delta)^{s}m(t,x)\le- C_4\partial_xw(t,x)=-C_4 v_0'(x) v_0^{-\beta}(x)w^{\beta}(t,x). \]

{\bf Case II :} $ 0<s\le\frac{1}{2} $. In this situation, again the previous argumentation fails and we argue as follows. For any $R>1$, let us rewrite $I_5$  as follows
\begin{eqnarray*}
	I_5&=&\int_{\frac{1}{2}}^{R}\frac{m(t,x)-m(t,x+z) }{|z|^{1+2s}}dz+\int_{R}^{+\infty}\frac{m(t,x)-m(t,x+z) }{|z|^{1+2s}}dz\\
	&=:&I_6+I_7.
\end{eqnarray*}

An easy computation show that since  $m(t,x)\leq w(t,x)$, and $R>1$ we get 
\begin{eqnarray}
	I_7\le  \int_{R}^{+\infty}\frac{m(t,x) }{|z|^{1+2s}}dz&\leq&  w(t,x)\cdot \int_R^{+\infty} \frac{1}{z^{1+2s}} dz\nonumber\\
	&\le& w(t,x)\cdot \int_R^{+\infty} \frac{1}{z^{1+s}} dz= \frac{1}{s} w(t,x)\cdot  \frac{1}{R^{s}}.\label{com2-4}
\end{eqnarray}

On the other hand by using the Fundamental Theorem of Calculus, we obtain the following for $I_6$
\begin{eqnarray*}
	I_6&=&-\int_{\frac{1}{2}}^{R}\int_0^1\frac{z\partial_x w(t,x+\tau z)\left(1 -2Bw(t,x+\tau z)\right)}{|z|^{1+2s}} \,dzd\tau,
\end{eqnarray*}
and  by using the convexity and the monotonicity of $w(t,\cdot)$,   we then achieve 
\begin{eqnarray}
	I_6&\leq & -\int_{\frac{1}{2}}^{R}\int_0^1\frac{z}{|z|^{1+2s}}  \partial_x w(t,x+\tau z)\left(1 -2Bw(t,x+\tau z)\right) \,dzd\tau,  \nonumber\\
	   &\le  & -\partial_x w(t,x)\int_{\frac{1}{2}}^{R}\frac{z}{|z|^{1+2s}}\,dz, \nonumber\\
	&\le & -\partial_x w(t,x)\left(\int_{\frac{1}{2}}^{1}\frac{1}{z^{2s}}\,dz+\int_{1}^{R}\frac{1}{z^{2s}}\,dz\right),\nonumber\\
	&\le & -\partial_x w(t,x)\left(\int_{\frac{1}{2}}^{1}\frac{1}{z^{2s}}\,dz+\int_{1}^{R}\frac{1}{z^{s}}\,dz\right)\le  -\partial_x w(t,x)\left(C(s)+\frac{1}{s}R^{1-s}\right).\label{com2-5}
\end{eqnarray}


%
%
Combining \eqref{com2-5}  with \eqref{com2-4} and using the definition of $\partial_xw(t,x)$ enforce that
$$
	I_5\leq \frac{1}{s} w(t,x)\cdot  \frac{1}{R^{s}}+C_{6,2}[w(t,x)]^\beta\cdot x^{2s\beta-2s-1}R^{1-s} +C(s)\partial_xw(t,x).
$$

By taking $R$ such that $C_{6,2}[w(t,x)]^\beta\cdot x^{2s\beta-2s-1}R^{1-s}=\frac{1}{s} w(t,x)\cdot  \frac{1}{R^{s}}$, that is,  $R=\frac{1}{s C_{6,2} }\cdot [w(t,x)]^{1-\beta} \cdot x^{-2s\beta+2s+1}$, we then  get 
$$
	I_5 \leq C_{7,2} [w(t,x)]^{1-2s+2s\beta}x^{2s(2s\beta-2s-1)}+ C(s)\partial_xw(t,x),
$$
which combined with  \eqref{com2-1}, \eqref{com2-2} and \eqref{com2-3} then yields
$$
	 (-\Delta)^{s}m(t,x)\le -C_4\partial_xw(t,x)+C_4[w(t,x)]^{1-2s+2s\beta} x^{2s(2s\beta-2s-1)}.
$$

%
%
%

\end{proof}

Lastly, let us estimate $-(-\Delta)^{s}m(t,x)$ in the region  $x_B(t)-1<x<x_B(t)+1$.

\subsubsection*{\underline{$\bullet$ When $x_B(t)-1\le x\le x_B(t)+1$}:}~\\ 
In this last region, we claim that:
\begin{claimd}\label{subsolcla3}
    \begin{enumerate}
    	\item[(a)] If $\frac{1}{2}<s<1$,  then there exists positive constant $C_5$ such that for all $ x_B(t)-1\le x\le x_B(t)+1$, we have
    	$$ (-\Delta)^{s}m(t,x)\le -C_5 v_0'(x) [v_0(x)]^{-\beta}[w(t,x)]^{\beta}.$$ 
    	
    	\item[(b)] If $0<s\le \frac{1}{2}$, for large enough $B\gg1$, then there exists positive constant $C_5$ such that for all $ x_B(t)-1\le x\le x_B(t)+1$, , we have
    	$$ (-\Delta)^{s}m(t,x)\le -C_5\partial_xw(t,x)+C_5[w(t,x)]^{1-2s+2s\beta} x^{2s(2s\beta-2s-1)}.$$ 
    	
    	
    \end{enumerate}

\end{claimd}

\begin{proof}
	Again let us rewrite the fractional Laplacian in the following way : 
	\begin{eqnarray*}
		(-\Delta)^{s}m(t,x)&=&\textnormal{P.V. }\int_{\mathbb{R}} \frac{m(t,x)-m(t,y)}{|x-y|^{1+2s}} dy\\
		&=&\int_{-\infty}^{x-1} \frac{m(t,x)-m(t,y)}{|x-y|^{1+2s}} dy+\textnormal{P.V. }\int_{x-1}^{+\infty} \frac{m(t,x)-m(t,y)}{|x-y|^{1+2s}} dy\\
		&=:&I_1+I_2
	\end{eqnarray*}
	
	By using the monotone character of $m(t,\cdot)$, we have 
	\begin{eqnarray}
		I_1\leq0.\label{com3-1}
	\end{eqnarray}
	
	For $I_2$, we have
	\begin{align*}
	I_2& =  \textnormal{P.V. }\int_{-1}^{\infty}\frac{m(t,x)-m(t,x+z)}{|z|^{1+2s}}dz \\    
	& =- \frac{1}{2} \int_{-1}^{1}\frac{m(t,x+z)+m(t,x-z)-2m(t,x)}{|z|^{1+2s}}dz + \int_{1}^{\infty}\frac{m(t,x)-m(t,x+z)}{|z|^{1+2s}}dz \\
	&= I_3+I_4.  
	\end{align*}
	Observe that again to estimate $I_2$ we break the integral into two part and we can easily see that the contribution of $I_4$ can be estimated as in the proof of the previous claim so we won't repeat it. If fact, here the only change  with respect to the situation $x>x_B(t)+1$, is the contribution of $I_3$ since  unlike the previous case, the function $m$ is not any more a $C^2$ smooth function on the domain of integration and we need then more precise estimate. So let us now look more closely at $I_3$.
	  
	For $I_3$, thanks to the definition of $m$, we  can get
	\begin{align*}
	I_3&=-\frac{1}{2} \int_{-1}^1\int_{0}^{1}z\frac{\partial_xm(t,x+\tau z) -\partial_xm(t,x-\tau z)}{|z|^{1+2s}}d\tau  dz\\
	&=-\frac{1}{2}\int_{-1}^1\int_{0}^{1}\frac{z}{|z|^{1+2s}}[\partial_x w(t,x+\tau z)(1-2Bw(t,x+\tau z))^+ -\partial_x w(t,x-\tau z)(1-2Bw(t,x-\tau z))^+ ]\,d\tau  dz
	\end{align*}
	
	Let us rewrite the bracket inside the integral as follows:
	\begin{eqnarray*}
		&& \partial_x w(t,x+\tau z)(1-2Bw(t,x+\tau z))^+ -\partial_x w(t,x-\tau z)(1-2Bw(t,x-\tau z))^+\\
		&=&	[\partial_x w(t,x+\tau z)- \partial_x w(t,x-\tau z)]\cdot (1-2Bw(t,x+\tau z))^+\\
		&& \quad +\, \partial_x w(t,x-\tau z)[(1-2Bw(t,x+\tau z))^+-(1-2Bw(t,x-\tau z))^+ ]
	\end{eqnarray*}
	
	Then we can decompose $I_3$ into two integrals 
	$ I_3=I_5+I_6$
	with 
	\begin{align*}
	&I_5:=\frac{1}{2}\int_{-1}^1\int_{0}^{1}\frac{z}{|z|^{1+2s}}[\partial_x w(t,x+\tau z)- \partial_x w(t,x-\tau z)] (1-2Bw(t,x+\tau z))^+ \,d\tau dz\\
	&I_6:=\frac{1}{2} \int_{-1}^1\int_{0}^{1}\frac{z}{|z|^{1+2s}} \partial_x w(t,x-\tau z)[(1-2Bw(t,x+\tau z))^+-(1-2Bw(t,x-\tau z))^+ ]\,d\tau  dz.
	\end{align*}
	
	Now since $w(t,x)$ is smooth, using the Fundamental Theorem of Calculus, we get 
	\begin{eqnarray*}
		I_5=-\frac{1}{2}\int_{-1}^1\int_{0}^{1}\int_{-1}^1\frac{2\tau z^2}{|z|^{1+2s}}\partial_{xx} w(t,x+\sigma \tau z) (1-2Bw(t,x+\tau z))^+ \,d\sigma d\tau dz.
	\end{eqnarray*}

    Since $w(t,\cdot)$ is convex, then
    \begin{eqnarray*}
    	I_5\leq0. 
    \end{eqnarray*}

	For $I_6$, by using the convexity of $w(t,\cdot)$ and the uniform Lipschitz continuity of the function $(1-2Bw(t,x))^+$ for  $x \in (x_B(t)-1,x_B(t)+1)$, then we have
	\begin{align*}
	I_6&\le -\frac{1}{2} \partial_xw(t,x-1)\int_{-1}^1\int_{0}^{1}\frac{|z|}{|z|^{1+2s}} |(1-2Bw(t,x+\tau z))^+-(1-2Bw(t,x-\tau z))^+ |\,d\tau  dz\\
	&\le -C\partial_xw(t,x-1)\int_{-1}^{1}\frac{z^2}{|z|^{1+2s}}dz \\
	&=-C_{5,2}\partial_x w(t,x-1).
	\end{align*}
	
	A direct computation can give us some $C_{5,1}>0$ such that  $\partial_xw(t,x-1)\ge C_{5,1}\partial_x w(t,x)$, which implies that
	$$I_6\le -C_{5,3}\partial_x w(t,x).$$
	
	Hence
	\begin{eqnarray}
		I_3\leq -C_{5,3}\partial_x w(t,x).\label{com3-2}
	\end{eqnarray}

\end{proof}

By collecting  \eqref{eq-dtm}, {\bf Claim \ref{subsolcla1}}, {\bf Claim \ref{subsolcla2}}, {\bf Claim \ref{subsolcla3}} and \eqref{ac-eq-f}, now we can show that $m(t,x)$ is a subsolution for some appropriate choices of $B$, $\gamma$ and $\epsilon$.

If $\frac{1}{2}<s<1$, by \eqref{eq-dtm}, {\bf Claim \ref{subsolcla1}}, {\bf Claim \ref{subsolcla2}}, {\bf Claim \ref{subsolcla3}} and \eqref{ac-eq-f}, we have
$$
(\partial _t m +(-\Delta)^{s}m-f(m))(t,x) \leq 
 \begin{cases}
-[w(t,x_B(t))]^{\beta}\left[C_0+h(t,x_B(t))\right] &\text{ for }  x \le x_B(t)-1\\
 -[w(t,x)]^{\beta}\left[C_0+h(t,x) -\gamma\right] &\text{ for } x > x_B(t)-1,
\end{cases}
$$
where  $h(t,x)= C_6v_0'(x)[v_0(x)]^{-\beta}$ with $C_6\ge \max\{C_3,C_4,C_5\}$.

We now choose $\gamma\le \frac{C_0}{2}$. In the view of the above inequalities, to complete the construction of the subsolution $m(t,x)$, it suffices to find a condition on $B$ so that $h(t,x)\geq -\frac{C_0}{2}$ for all $t>0$ and all $x\in \R$. From the definitions of $h(t,x)$ and $v_0(x)$, it suffices to achieve
$$
x^{(\beta-1)2s-1}\le  \frac{C_0d^{\beta-1}}{4s dC_6 }, \quad \text{ for all } t>0, x\geq x_B(t)-1.
$$

Since $(\beta-1)2s<1$, this reduces to the following condition on $x_B(0)$:
\begin{equation*} 
x_B(0)\ge \left(\frac{C_0d^{\beta-1}}{4s dC_6}\right)^{\frac{1}{1-2s(\beta-1)}}+1. \label{ac-x0t-ineq}  
\end{equation*}

From \eqref{ac-def-xB} we have $x_B(0)=(2Bd)^{\frac{1}{2s}}$. Hence, in view of the definition of $C_0$, the above inequality holds by selecting $B\geq B_0$, with $B_0>0$ large enough.

If $0<s\le \frac{1}{2}$, when $x\leq x_B(t)-1$, by  \eqref{eq-dtm}, {\bf Claim \ref{subsolcla1}}, and \eqref{ac-eq-f}, we have
\begin{eqnarray*}
	(\partial _t m +(-\Delta)^{s}m-f(m))(t,x)&\leq&\frac{C_3}{B^2}-C_0[w(t,x_B(t))]^{\beta}\\
	&=&\frac{C_3}{B^2}-C_0\left( \frac{1}{2B} \right)^\beta\\
	&=&B^{-2}[C_3-2^{-\beta} C_0B^{2-\beta}   ]
\end{eqnarray*}

Since $\beta<2$, then there exists some $B_1\gg1$ such that $C_3-2^{-\beta} C_0B^{2-\beta} <0$ for all $B\geq B_1$. Hence in this case, we have
\[ (\partial _t m +(-\Delta)^{s}m-f(m))(t,x)<0. \]

When $x>x_B(t)-1$, by  \eqref{eq-dtm}, {\bf Claim \ref{subsolcla2}}, {\bf Claim \ref{subsolcla3}} and \eqref{ac-eq-f}, we have
\begin{eqnarray*}
	&&(\partial _t m +(-\Delta)^{s}m-f(m))(t,x)\\
	&\leq& \gamma [w(t,x)]^{\beta}-C_6\partial_xw(t,x)+C_6[w(t,x)]^{1-2s+2s\beta} x^{2s(2s\beta-2s-1)}-C_0w^{\beta}(t,x)\\
	&=&\gamma [w(t,x)]^{\beta}+C_7[w(t,x)]^\beta\cdot x^{2s\beta-2s-1}+C_6[w(t,x)]^{1-2s+2s\beta} x^{2s(2s\beta-2s-1)}-C_0[w(t,x)]^{\beta}\\
	&=&[w(t,x)]^{\beta}[\gamma+C_7 x^{2s\beta-2s-1}+C_6[w(t,x)]^{1-2s+2s\beta-\beta}x^{2s(2s\beta-2s-1)}  -C_0   ]
\end{eqnarray*}

Take $\gamma=\frac{C_0}{3}$. It's easy to see that $x_B(t)\geq d^{\frac{1}{2s}} (2B)^{\frac{1}{2s}}$. Since $2s\beta-2s-1<0$, so when $B$ is large enough, we have 
\[ C_7 x^{2s\beta-2s-1}\leq \frac{C_0}{3}. \]

Note that since $\beta>1$ and $0<s\le \frac{1}{2}$, we have  $1-2s+2s\beta-\beta\le 0$ and  therefore since $w(t,x)\geq w(0,x)=v_0(x)=\frac{d}{x^{2s}}$, we have 
\begin{eqnarray*}
	C_6[w(t,x)]^{1-2s+2s\beta-\beta}x^{2s(2s\beta-2s-1)} &\leq& C_8 x^{-2s(1-2s+2s\beta-\beta)}\cdot x^{2s(2s\beta-2s-1)}\\
	&=&C_8x^{2s(\beta-2)}
\end{eqnarray*}

Using that  $\beta<2$ and since $x\geq x_B(t)-1\geq d^{\frac{1}{2s}} (2B)^{\frac{1}{2s}}-1  $, so when $B$ is large enough, we have $C_8x^{2s(\beta-2)}\leq \frac{C_0}{3}$, which implies that 
\[C_6[w(t,x)]^{1-2s+2s\beta-\beta}x^{2s(2s\beta-2s-1)} \leq \frac{C_0}{3}.\]

So 
\[ (\partial _t m +(-\Delta)^{s}m-f(m))(t,x)\leq0. \]

In summary, for any $0<s<1$, after some good choices of $\gamma$, $B$ and $\epsilon$, then the function $m(t,x)$ indeed is a subsolution.

\medskip
\noindent {\bf Step three.} It consists in using the subsolution to prove the lower estimate in {\bf Theorem \ref{thm:1-3}}.

Fix $\gamma>0$ and $B_0 >0$ as in the previous step so that $m(t,x)$ is a subsolution. From the comparison principle we get $m(t,x)\le u(t,x)$, for all $t>0$ and $x \in \R$. Recall that $m(t,x_{B_0}(t))=\frac{1}{4B_0}$ and that $u(t,\cdot)$ is nonincreasing (since initial datum $v_0$ is nonincreasing) so that
\begin{equation}
\label{u-grand}
u(t,x)\geq \frac{1}{4B_0}, \quad \forall x\leq x_{B_0}(t).
\end{equation}
In particular, for any $0<\lambda \leq \frac{1}{4B_0}$, the \lq\lq largest'' element $x_\lambda(t)$ of the super level set $\Gamma _\lambda (t)$ has to satisfy
$$
x_\lambda(t)\geq x_{B_0}(t)\geq d^{\frac{1}{\alpha -1}}[\gamma (\beta -1)t]^{\frac{1}{2s(\beta-1)}},
$$
which provides the lower estimate.

It now remains to obtain a similar bound for a given $\frac{1}{4B_0}<\lambda <1$.  Such estimate  can be obtained by redoing the argument in {\bf Section \ref{sec-l}}.

\medskip

\noindent \textbf{Acknowledgement.}
J.Coville  carried out this work in the framework of Archim\`ede Labex (ANR-11-LABX-0033) and of the A*MIDEX project (ANR-11-IDEX-0001-02), funded by the ``Investissements d'Avenir" French Government program managed by the French National Research Agency (ANR). He has also received funding from the ANR DEFI project NONLOCAL (ANR-14-CE25-0013). C. Gui is supported by NSF grants DMS-1601885 and DMS-1901914 and Simons Foundation  Award 617072.  M. Zhao is supported by the National Natural Science Foundation of China No.11801404.
\bibliographystyle{plain}
\bibliography{biblio-asym}

\end{document}